\pdfoutput=1
\documentclass[11pt, reqno]{amsart}
\usepackage[dvips]{graphicx}
\usepackage{bmpsize} 
\usepackage[margin=1.12in]{geometry}

\usepackage{mathrsfs}
\usepackage{amsmath,amscd}
\usepackage{amssymb}
\usepackage{amscd}
\graphicspath{/home/jhois/Documents/Images/}  

\newtheorem{theorem}{Theorem}[section]
\newtheorem{proposition}[theorem]{Proposition}
\newtheorem{lemma}[theorem]{Lemma}
\newtheorem{corollary}[theorem]{Corollary}

\newtheorem{definition}[theorem]{Definition}

\theoremstyle{definition} 

\newtheorem{remark}[theorem]{Remark}

\newcommand{\C}{\mathbb C} 
\newcommand{\R}{\mathbb R} 
\newcommand{\Z}{\mathbb Z}

\numberwithin{equation}{section} 
\numberwithin{theorem}{section}
\numberwithin{figure}{section}

\begin{document}

\author[J.A.~Hoisington]{Joseph Ansel Hoisington} \address{Department of Mathematics, University of Georgia, Athens, GA 30602 USA}\email{jhoisington@uga.edu}

\title[Calibrations and Energy-Minimizing Maps]{Calibrations and Energy-Minimizing Mappings of Rank-1 Symmetric Spaces}


\begin{abstract}
We prove lower bounds for energy functionals of mappings from real, complex and quaternionic projective spaces to Riemannian manifolds.  For real and complex projective spaces, these lower bounds are sharp, and we characterize the family of energy minimizing maps which arise in these results.  We discuss the connections between these results and several theorems and questions in systolic geometry. 
\end{abstract}

\maketitle



\section{Introduction} 
\label{introduction} 

\subsection*{Statement of Results:} The main results of this work are lower bounds for energy functionals of mappings from real, complex and quaternionic projective spaces to Riemannian manifolds and characterizations of mappings which minimize energy in these results.  For real projective space, we will prove:  

\begin{theorem}
\label{rpn_thm}
Let $(\R P^{n},g_{0})$ be the $n$-dimensional real projective space with its canonical Riemannian metric $g_{0}$ of constant curvature $1$, with $n \geq 2$.  Let $(M,g)$ be a Riemannian manifold and $F: (\R P^{n}, g_{0}) \rightarrow (M^{m},g)$ a Lipschitz mapping.  Let $L^{\star}$ be the infimum of the lengths of paths in the free homotopy class of $F_{*}(\gamma)$, where $\gamma$ represents the non-trivial class in $\pi_{1}(\R P^{n})$, and let $E_{p}(F)$ be the $p$-energy of $F$, as in Definition \ref{p_energy_def} below. \\ 

Then for all $p \geq 1$, 

\begin{equation}
\label{rpn_thm_eqn}
\displaystyle E_{p}(F) \geq \frac{\sigma(n)}{2}\left(\frac{2\sqrt{n}}{\pi}L^{\star}\right)^{p}, \medskip 
\end{equation}
where $\sigma(n)$ is the volume of the unit $n$-sphere. \\ 

If $p > 1$ and equality holds for $p$, then $F$ is a homothety onto a totally geodesic submanifold of $(M,g)$ and equality holds for all $p \geq 1$.  If $F$ is a smooth immersion, equality for $p=1$ also implies these conditions. 
\end{theorem}

This implies in particular that the identity mapping of $(\R P^{n},g_{0})$ minimizes $p$-energy in its homotopy class for all $p \geq 1$. \\

The case $p=2$ of Theorem \ref{rpn_thm} was proven by Croke \cite[Theorem 1]{Cr1}.  The $2$-energy, referred to simply as the energy of a mapping, is the classical energy functional of mappings of Riemannian manifolds and a fundamental invariant in the theory of harmonic maps.  It is a generalization of the Dirichlet integral of a real-valued function and the energy of a path in a Riemannian manifold.  The fact that the identity mapping of $(\R P^{n},g_{0})$ minimizes energy in its homotopy class was first established by Croke as a corollary of this result.  This is in contrast to the round $n$-sphere $(S^{n},g_{0})$, $n \geq 3$, for which the identity mapping is not energy minimizing in its homotopy class.  Conformal dilations give energy decreasing deformations of the identity mapping of $(S^{n},g_{0})$, $n \geq 3$, in fact the infimum of the energy over this family of mappings is $0$.  More generally, work of White \cite{Wh1} implies that in any closed Riemannian manifold $(M,g)$ with $\pi_{1}(M) =  \pi_{2}(M) = 0$ the identity mapping is homotopic to maps with arbitrarily small energy.  Croke also noted in \cite{Cr1} that the results of Smith \cite{Sm1} imply there are metrics arbitrarily close to the canonical metric on $\R P^{n}$, obtained by conformal deformations, for which the identity mapping is not even a stable critical point of the energy functional. \\ 

Croke observed in \cite[Theorem 3]{Cr1} that his argument for mappings of real projective space could be adapted to establish a lower bound for the energy of mappings of $\C P^{N}$ with its canonical metric $g_{0}$, and that this lower bound implies the identity mapping of $(\C P^{N},g_{0})$ minimizes energy in its homotopy class, but that this was already known because $(\C P^{N}, g_{0})$ is a K\"ahler manifold -- in fact, Lichnerowicz established in \cite{Li1} that any holomorphic mapping of compact K\"ahler manifolds minimizes energy in its homotopy class. \\ 

A natural extension of Theorem \ref{rpn_thm} to complex projective space would give lower bounds for the $p$-energy of mappings $F: (\C P^{N}, g_{0}) \rightarrow (M^{m},g)$ in terms of the infimum $A^{\star}$ of the areas of mappings $f:S^{2} \rightarrow M$ which represent the homotopy or homology class of $F_{*}(\C P^{1})$, where $(M,g)$ is a Riemannian manifold.  Unlike the strong characterization of equality in Theorem \ref{rpn_thm} however, basic properties K\"ahler manifolds imply that in any such optimal result, the equality case for the classical energy functional must be broad enough to include any holomorphic mapping from $\C P^{N}$ to a compact K\"ahler manifold -- we will explain this in detail at the beginning of Section \ref{complexes}.  Also, although conformal deformations of the canonical metric give Riemannian metrics on $\C P^{N}$ for which the identity mapping is not a stable critical point of the energy functional, as with $\R P^{n}$, on $\C P^{N}$ Lichnerowicz's theorem cited above also gives an infinite-dimensional family of metrics, obtained by K\"ahler deformations of the canonical metric, for which the identity mapping is energy minimizing in its homotopy class. \\ 

In Theorem \ref{cpn_thm} we will state and prove lower bounds for the $p$-energy of Lipschitz mappings $F:(\C P^{N},g_{0}) \rightarrow (M,g)$, $p \geq 2$, where $(M,g)$ is a Riemannian manifold.  The full characterization of equality in this result is somewhat technical and involves several partial results under weaker assumptions, but for the complex projective plane, our results imply that holomorphic mappings are essentially the only energy minimizing maps of this type.  We record this in the following:   

\begin{theorem}
\label{kahler_thm} 
Let $F:(\C P^{2}, g_{0}) \rightarrow (M^{m}, g)$ be a Lipschitz mapping to a Riemannian manifold.  Let $A^{\star}$ be the infimum of the areas of Lipschitz mappings $f:S^{2} \rightarrow (M,g)$ in the free homotopy class of $F_{*}(\C P^{1})$, and let $E_{p}(F)$ be the $p$-energy of $F$. \\ 

Then for all $p \geq 2$, 

\begin{equation}
\label{kahler_thm_eqn}
\displaystyle E_{p}(F) \geq \frac{\pi^{2}}{2} \left( \frac{4}{\pi} A^{\star} \right)^{\frac{p}{2}}. \medskip 
\end{equation}

If equality holds for $p = 2$ then $F$ is smooth and $F^{*}g$ is $U(1)$-invariant, where $U(1)$ represents the unit complex numbers acting on the tangent bundle $T\C P^{2}$.  Letting $\mathcal{V} \subseteq \C P^{2}$ be the domain on which $rk(dF) = 4$, $F^{*}g|_{\mathcal{V}}$ is a K\"ahler metric, $F(\mathcal{V})$ is minimal in $(M,g)$ and the second fundamental form of $F(\mathcal{V})$ in $(M,g)$ can be diagonalized by a unitary basis. \\ 

If $p > 2$ and equality holds for $p$, then $F$ has constant energy density and equality holds for all $p \geq 2$.  If $F$ is an immersion and equality holds for some $p > 2$, then $F$ is a homothety onto its image. 
\end{theorem}

It would be interesting to determine whether Theorem \ref{kahler_thm} is true for mappings of $\C P^{N}$, $N \geq 3$.  The proof of Theorem \ref{kahler_thm} uses a lemma which gives the following local characterization of K\"ahler surfaces -- the author does not know of a reference for this fact:  

\begin{proposition}
\label{kahler_characterization}
Let $(X,h)$ be a Hermitian surface (of complex dimension $2$).  Suppose that for all $x_{0} \in X$ and all complex lines $\Pi$ in $T_{x_{0}}X$ there is a complex curve $\Sigma_{\Pi} \subseteq X$ (of complex dimension $1$) with $\Pi$ tangent to $\Sigma_{\Pi}$, and with the mean curvature of $\Sigma_{\Pi}$ vanishing at $x_{0}$.  Then $h$ is a K\"ahler metric. \\ 

In particular, if $(X,h)$ is a Hermitian surface in which all complex curves are minimal then $h$ is a K\"ahler metric. 
\end{proposition}

Proposition \ref{kahler_characterization} is a special case of a more general result which we state and prove in Lemma \ref{kahler_form_prop}. \\ 

For some mappings of $\C P^{N}$ to compact, simply connected K\"ahler manifolds, we will prove a stronger characterization of the equality case in Theorem \ref{cpn_thm} than holds in general -- we record this equality case, along with our general lower bound for the $p$-energy of mappings of $\C P^{N}$, in the following:  

\begin{theorem}
\label{cpn_holom_thm}
Let $F:(\C P^{N},g_{0}) \rightarrow (M,g)$ be a Lipschitz mapping to a Riemannian manifold and $A^{\star}$ the infimum of the areas of Lipschitz mappings $f:S^{2} \rightarrow M$ in the free homotopy class of $F_{*}(\C P^{1})$. \\ 

Then for all $p \geq 2$, 

\begin{equation}
\label{general_cpn_estimate}
\displaystyle E_{p}(F) \geq \frac{\pi^{N}}{N!} \left( \frac{2N}{\pi} A^{\star} \right)^{\frac{p}{2}}. \medskip 
\end{equation}

Suppose in addition that $(M,g)$ is a compact, simply connected K\"ahler manifold, and that the class of $F_{*}(\C P^{1})$ in $H_{2}(M;\Z)$ can be represented by a rational curve, that is, by a holomorphic mapping $f:\C P^{1} \rightarrow M$.  Then if equality holds for $p=2$, $F$ is holomorphic.  If $p > 2$ and equality holds for $p$, then $F$ is a homothety onto its image and equality holds for all $p \geq 2$. \\ 

Likewise, if $F_{*}(\C P^{1}) \in H_{2}(M;\Z)$ can be represented by an antiholomorphic mapping $f:\C P^{1} \rightarrow M$, equality for $p=2$ implies $F$ is antiholomorphic and equality for $p > 2$ implies $F$ is a homothety.  
\end{theorem}

The equality conditions in Theorems \ref{kahler_thm}, \ref{cpn_holom_thm} and \ref{cpn_thm} are related to those in several results of Ohnita \cite{Oh1} and Burns, Burstall, de Bartolomeis and Rawnsley \cite{BBdBR1}, which we will discuss in Section \ref{complexes}. \\  

White's results \cite{Wh1} cited above imply that the identity mapping of $(\C P^{N},g_{0})$ is homotopic to maps with arbitrarily small $p$-energy for all $ 1 \leq p < 2$.   They likewise imply that the identity mapping of quaternionic projective space $\mathbb{H} P^{N}$ with its canonical metric $g_{0}$ is homotopic to maps with arbitrarily small $p$-energy for all $1 \leq p < 4$ and the identity mapping of the Cayley projective plane with its canonical metric $(\mathcal{C}a P^{2},g_{0})$ is homotopic to maps with arbitrarily small $p$-energy for all $1 \leq p < 8$.  In this sense, Theorems \ref{kahler_thm} and \ref{cpn_holom_thm} are optimal, and the strongest results one can hope to establish in $(\mathbb{H} P^{N},g_{0})$ are lower bounds for the $p$-energy of mappings for $p \geq 4$. \\ 

In some ways, a result of this type for $\mathbb{H} P^{N}$ would be a natural extension of Theorems \ref{rpn_thm}, \ref{kahler_thm} and \ref{cpn_holom_thm} -- we will discuss this below -- but the optimal result for $\mathbb{H} P^{N}$ cannot be as strong as our results for $\R P^{n}$ and $\C P^{N}$:  the work of Wei \cite{We1} implies that for $N \geq 2$, the identity mapping of $(\mathbb{H} P^{N},g_{0})$ is not a stable critical point of the $4$-energy and in particular does not minimize $4$-energy in its homotopy class.  We will show that for $p \geq 4$, the $p$-energies of mappings of $(\mathbb{H} P^{N},g_{0})$  nonetheless do satisfy lower bounds similar to our results for mappings of $(\R P^{n},g_{0})$ and $(\C P^{N},\widetilde{g})$ above: 

\begin{theorem}
\label{hpn_thm}
Let $(\mathbb{H}P^{N},g_{0})$ be the quaternionic projective space with its canonical Riemannian metric $g_{0}$ normalized to have sectional curvature $K$ with $1 \leq K \leq 4$, with $N \geq 2$.  Let $F:(\mathbb{H}P^{N},g_{0}) \rightarrow (M^{m},g)$ be a non-constant Lipschitz map to a closed Riemannian manifold, and let $B^{\star}$ be the minimum mass of an integral 4-current $T$ in $M$ which represents the class of $F_{*}(\mathbb{H}P^{1})$ in $H_{4}(M;\Z)$. \\ 

Then for all $p \geq 4$,  
	
\begin{equation}
\label{hpn_thm_eqn}
\displaystyle E_{p}(F) > \frac{\pi^{2N}}{(2N+1)!} \left( K_{N}B^{\star} \right)^{\frac{p}{4}}, \medskip 
\end{equation}
where $K_{N}$ is a positive constant which depends only on $N$ and is given in (\ref{hpn_const_formula}) below.
\end{theorem}

The minimum $B^{\star}$ in Theorem \ref{hpn_thm} is over all integral currents homologous to $F_{*}(\mathbb{H}P^{1})$, which may be larger than the set of mappings $f:S^{4} \rightarrow M$ homotopic to $F_{*}(\mathbb{H}P^{1})$ as in Theorems \ref{rpn_thm}, \ref{kahler_thm} and \ref{cpn_holom_thm}.  However a theorem of White \cite{Wh2} shows that if $M$ is simply connected, $B^{\star}$ is equal to the infimum of the areas of mappings $f:S^{4} \rightarrow M$ in the free homotopy class of $F_{*}(\mathbb{H}P^{1})$.  We will discuss this after the proof of Theorem \ref{hpn_thm}, in Section \ref{quaternions}.  Although the identity mapping of $(\mathbb{H} P^{N},g_{0})$ does not minimize $4$-energy in its homotopy class, the proof of Theorem \ref{hpn_thm} does show that the identity mapping is $4$-energy minimizing among maps which satisfy an additional hypothesis -- we will also discuss this after the proof of Theorem \ref{hpn_thm}. \\ 

The strongest conjecture for mappings of $(\mathbb{H}P^{N},g_{0})$ which is consistent with Wei's result \cite[Theorem 5.1]{We1} in all dimensions is that the identity mapping of $(\mathbb{H} P^{N},g_{0})$ minimizes $p$-energy in its homotopy class for $p \geq 6$.  More precisely, Wei's results imply that the identity mapping of $(\mathbb{H} P^{N},g_{0})$ is an unstable critical point of the $p$-energy for $1 \leq p < 2 + 4(\frac{N}{N+1})$ and a stable critical point for $p \geq 2 + 4(\frac{N}{N+1})$.  For the Cayley plane $(\mathcal{C}a P^{2},g_{0})$ Wei's results imply that the identity mapping is an unstable critical point of the $p$-energy for $1 \leq p < 10$ and a stable critical point for $p \geq 10$.  It would be interesting to determine whether Theorem \ref{hpn_thm} gives an optimal lower bound for the $p$-energy of mappings of $(\mathbb{H} P^{N},g_{0})$.  More generally, it would be interesting to find optimal lower bounds for energy functionals of mappings of $(\mathbb{H} P^{N},g_{0})$ and $(\mathcal{C}a P^{2},g_{0})$ and study $p$-energy minimizing or approximately $p$-energy minimizing mappings of these spaces.  It would also be interesting to determine for which $p$ the identity mappings of $(\mathbb{H} P^{N},g_{0})$ and $(\mathcal{C}a P^{2},g_{0})$ are $p$-energy minimizing in their homotopy classes.  At the end of Section \ref{quaternions}, we will sketch one possible approach to this problem for $\mathbb{H} P^{N}$.  \\ 

For mappings homotopic to the identity of $\C P^{N}$, the lower bound in Theorem \ref{cpn_holom_thm} follows from a lower bound for the area of surfaces homologous to $\C P^{1}$ in $\C P^{N}$.  This lower bound follows from the calibrated structure given by the K\"ahler $2$-form of $(\C P^{N},g_{0})$.  The characterization of energy minimizing maps as holomorphic or antiholomorphic in Theorem \ref{cpn_holom_thm} is related to the fact that maps which realize the lower bound in (\ref{general_cpn_estimate}) must map linearly embedded $\C P^{1} \subseteq \C P^{N}$ to cycles which are calibrated by the K\"ahler form of their target.  We will explain at the end of Section \ref{reals} how the canonical $1$-form on the unit tangent bundle of $(\R P^{n},g_{0})$ gives a calibration-like structure, and how the $1$-energy minimizing property of $Id:(\R P^{n},g_{0}) \rightarrow (\R P^{n},g_{0})$ in Theorem \ref{rpn_thm} can be derived from this construction.  Quaternionic projective space carries a calibration by a parallel $4$-form, described in \cite{Be2,Kr1,Kr2}, and this plays a part in the proof of Theorem \ref{hpn_thm}.  Despite these similarities, however, the fact that the identity mapping of $(\mathbb{H} P^{N},g_{0})$ does not minimize $4$-energy in its homotopy class, contrary to our results for the $1$-energy of mappings of $(\R P^{n},g_{0})$ and the $2$-energy of mappings of $(\C P^{N},g_{0})$, mirrors an important difference between the systolic geometry of the projective planes $\R P^{2}$, $\C P^{2}$ and $\mathbb{H} P^{2}$.  These systolic results have several connections to the results in this paper.  Our results also have some connections to the Blaschke conjecture, cf. \cite{Be1}.  We will finish this introduction by discussing the relationships between these results and giving an outline of the rest of the paper.  

\subsection*{Systolic Geometry and the Blaschke Conjecture:} Pu's inequality, the first published result in systolic geometry, gives a lower bound for the area of a Riemannian metric on the real projective plane in terms of the minimum length of its non-contractible curves: 

\begin{theorem}[Pu's Inequality, \cite{Pu}]
\label{pu_thm}
Let $g$ be a Riemannian metric on $\R P^{2}$.  Let $A(\R P^{2},g)$ be its area and $sys(g)$ its systole, that is, the minimum length of a non-contractible closed curve in $(\R P^{2},g)$.  Then: 

\begin{equation}
\label{pu_eqn}
\displaystyle A(\R P^{2},g) \geq \left( \frac{2}{\pi} \right) sys(g)^{2}.  \medskip 
\end{equation}

Equality holds if and only if $(\R P^{2},g)$ has constant curvature. 
\end{theorem}

In Section \ref{energies}, we will explain how Theorem \ref{pu_thm} follows from Croke's proof of the $p=2$ case of Theorem \ref{rpn_thm} in \cite{Cr1}.  We note that Croke's reasoning in the last section of \cite{Cr1} can also be used to show that the canonical metric on $\R P^{n}$ is infinitesimally optimal for an inequality of the form $Vol(\R P^{n},g) \geq G_{n}sys(g)^{n}$ for all $n \geq 3$.  The results of Gromov \cite{Gr3} imply that such an inequality holds with a positive constant $G_{n}$ for all $n \geq 2$, but for $n \geq 3$ the optimal constant $G_{n}$ is not presently known.  Gromov has also proven an inequality for complex projective space which is analogous to Pu's inequality, in terms of an invariant known as the stable $2$-systole.  For a Riemannian metric $g$ on $\C P^{N}$, this can be defined as follows:  let $\mu_{k}(g)$ be the minimum area in $(\C P^{N},g)$ of a $2$-dimensional current representing $k \in H_{2}(\C P^{N};\Z) \cong \Z$.  The stable $2$-systole $stsys_{2}(g)$ of $g$ is:  

\begin{equation}
\label{stable_2_systole}
\displaystyle stsys_{2}(g) = \lim\limits_{k \to \infty} \left( \frac{1}{k} \right) \displaystyle \mu_{k}(g). \bigskip  
\end{equation} 

\begin{theorem}[Gromov's Stable Systolic Inequality for $\C P^{N}$, \cite{Gr2}, see also \cite{Gr1,BKSW}]
\label{gromov_thm} 

Let $g$ be a Riemannian metric on $\C P^{N}$, let $Vol(\C P^{N},g)$ be its volume and $stsys_{2}(g)$ its stable $2$-systole as above.  Then:  

\begin{equation}
\label{gromov_eqn}
\displaystyle Vol(\C P^{N},g) \geq \frac{stsys_{2}(g)^{N}}{N!}. 
\end{equation}
\end{theorem}

As in Pu's inequality for $\R P^{2}$, equality holds for the canonical metric $g_{0}$ on $\C P^{N}$ in Theorem \ref{gromov_thm}.  Unlike the rigidity of the equality case in Pu's inequalty, however, equality also holds for all K\"ahler metrics on $\C P^{N}$.  As with the broader characterization of equality in Theorem \ref{cpn_holom_thm} for complex projective space, compared to Theorem \ref{rpn_thm} for real projective space, this follows from the fact that for any K\"ahler metric $\widetilde{g}$ on $\C P^{N}$, complex curves are calibrated and thus area minimizing in their homology classes. \\

A result analogous to Pu's and Gromov's inequalities (\ref{pu_eqn}) and (\ref{gromov_eqn}) holds for the quaternionic projective plane.  Like the result for $\C P^{N}$ in Theorem \ref{gromov_thm}, this inequality is in terms of the stable $4$-systole, which can be defined by a limit for $H_{4}(\mathbb{H}P^{2};\Z) \cong \Z$ as in (\ref{stable_2_systole}).  However unlike the results for $\R P^{2}$ and $\C P^{2}$ in Pu's and Gromov's inequalities, Bangert, Katz, Shnider and Weinberger have shown that the canonical metric on $\mathbb{H} P^{2}$ is not optimal for this inequality:  

\begin{theorem}[\cite{BKSW}]
Let $g$ be a Riemannian metric on $\mathbb{H}P^{2}$ and $stsys_{4}(g)$ the stable $4$-systole of $g$ (as defined in \cite{Gr2,BKSW}).  There is a positive constant $D_{2}$, independent of $g$, such that: 

\begin{equation}
\label{hpn_systolic_eqn}
\displaystyle Vol(\mathbb{H}P^{2},g) \geq D_{2} stsys_{4}(g)^{2}. \medskip  
\end{equation}
The optimal constant in (\ref{hpn_systolic_eqn}) satisfies $\frac{1}{6} \geq D_{2} \geq \frac{1}{14}$, which excludes the value $\frac{3}{10}$ of the canonical metric.  
\end{theorem}

The proof of the characterization of equality for $p=1$ in Theorem \ref{rpn_thm} is based on the characterization of the canonical metric on $\R P^{n}$ as the only Riemannian metric on $\R P^{n}$ for which the first conjugate locus of each point $x_{0}$ consists of a single point (in fact $x_{0}$ itself).  Blaschke conjectured that this was the case and it was proven by the combined work of Berger, Green, Kazdan and Yang, cf. \cite{Gn, Be1}.  The Blaschke conjecture therefore implies that, among mappings which are immersions, the equality case in Theorem \ref{rpn_thm} for $p = 1$ is the same as for $p > 1$.  This is different from the corresponding result for mappings of $(\C P^{N},g_{0})$:  there are holomorphic mappings of $\C P^{2}$ which are not isometries of the canonical metric.  These mappings minimize $2$-energy in their homotopy class, but by Theorem \ref{kahler_thm} they do not minimize $p$-energy for $p > 2$.  Therefore, the equality case for $p=2$ in our results for $\C P^{2}$ is strictly larger than for $p > 2$, even among diffeomorphisms.  We note that there are generalizations of the Blaschke conjecture for for $\C P^{N}$, $\mathbb{H} P^{N}$ and $\mathcal{C}a P^{2}$, described in \cite{Be1}, which are currently open.  

\subsection*{Outline and Notation:} In Section \ref{energies} we will define the $p$-energy of a mapping of Riemannian manifolds and establish some of its basic properties.  In Section \ref{reals} we will prove Theorem \ref{rpn_thm}.  In Section \ref{complexes} we will prove Theorem \ref{cpn_thm}, of which Theorem \ref{kahler_thm} is a special case, and Theorem \ref{cpn_holom_thm}.  In Section \ref{quaternions}, we will prove Theorem \ref{hpn_thm}.  The proof of Theorem \ref{hpn_thm} uses the twistor fibration $\Psi:\C P^{2N+1} \rightarrow \mathbb{H} P^{N}$, and we will discuss some background related to the twistor fibration at the beginning of Section \ref{quaternions}. \\ 

Throughout, we will write $\sigma(k)$ for the volume of the unit sphere in $\R^{k+1}$.  For a Riemannian metric $g$ on a manifold $M$, $U(M,g)$ will denote the unit tangent bundle of $M$ with the metric $g$ and $U_{p}(M,g)$ will denote its fibre at $p \in M$.  

\subsection*{Acknolwedgements:} I am very happy to thank Christopher Croke, Joseph H.G. Fu, Mikhail Katz, Frank Morgan and Michael Usher for their input and feedback about this work. 


\section{Energy Functionals of Mappings} 
\label{energies}


In this section, we will define the $p$-energy of a mapping.  We will prove two elementary results, Lemmas \ref{elementary_lemma} and \ref{energy_formula_lemma}, which we will use in the proof of our main theorems below. \\ 

The energy of a Lipschitz mapping $F:(M^{m},g) \rightarrow (N^{n},h)$ of Riemannian manifolds is:  

\begin{equation}
\label{energy_eqn}
\displaystyle E_{2}(F) = \int\limits_{M} |dF_{x}|^{2} dVol_{g}, \bigskip 
\end{equation}
where $|dF_{x}|$ is the Euclidean norm of $dF:T_{x}M \rightarrow T_{F(x)}N$.  We note that many authors define the energy to be one half the expression in (\ref{energy_eqn}). \\ 

There are many equivalent ways to define the energy of a mapping, discussed by Eells and Sampson in \cite{ES1}.  In this work, they initiated the study of the critical points of the energy as a functional on mappings from $(M,g)$ to $(N,h)$.  These critical mappings are known as harmonic maps and have many important connections to minimal submanifold theory, K\"ahler geometry and several other topics in differential geometry and analysis. \\ 

The energy in (\ref{energy_eqn}) fits naturally into a $1$-parameter family of functionals: 

\begin{definition}
\label{p_energy_def}
Let $F:(M^{m},g) \rightarrow (N^{n},h)$ be a Lipschitz mapping of Riemannian manifolds.  For $p \geq 1$, the $p$-energy of $F$ is: 

\begin{equation}
\label{p-energy_eqn}
\displaystyle \int\limits_{M} |dF_{x}|^{p} dVol_{g}.  
\end{equation}
\end{definition}

The pointwise quantity $|dF_x|^{p}$, where defined, is called the $p$-energy density.  We will denote this $e_{p}(F)_{x}$.  It will be helpful to note that wherever this can be defined, i.e. at all $x \in M$ at which $F$ is differentiable, $F^{*}h$ is a positive semidefinite, symmetric bilinear form on $T_{x}M$ which can be diagonalized relative to $g$.  Letting $e_{1}, e_{2}, \cdots, e_{m}$ be an orthonormal basis for $T_{x}M$ (relative to $g$) of eigenvectors for $F^{*}h$, we then have: 

\begin{equation}
\label{energy_frame_eqn}
\displaystyle |dF_{x}|^{2} = \sum\limits_{i=1}^{m} |dF(e_{i})|^{2}. \bigskip 
\end{equation}

This gives the following elementary lower bound for the $p$-energy of $F:(M^{m},g) \rightarrow (N^{n},h)$ for $p \geq dim(M)$: 

\begin{lemma}
\label{elementary_lemma}
Let $(M^{m},g)$ be a finite volume Riemannian manifold. Let $F:(M^{m},g) \rightarrow (N^{n},h)$ be a Lipschitz mapping, and define $Vol_{g}(M,F^{*}h)$ to be: 

\begin{equation}
\label{pullback_vol_eqn}
\displaystyle \int\limits_{M}|\det(dF_{x})| dVol_{g}, \medskip  
\end{equation}
where $det(dF_{x})$ is $0$ if $rk(dF_{x}) < m$ and is the determinant of $dF:T_{x}M \rightarrow dF(T_{x}) \subseteq T_{F(x)}N$ if $rk(dF_{x}) = m < n$. \\ 

Then for $p \geq m$, 

\begin{equation}
\label{g_dim_energy_eqn}
\displaystyle E_{p}(F) \geq m^{\frac{p}{2}} \frac{Vol_{g}(M,F^{*}h)^{\frac{p}{m}}}{Vol(M,g)^{\frac{p-m}{m}}}. \medskip 
\end{equation}

For $p=m$, equality holds if and only if $dF_{x}$ is a homothety at almost all $x \in M$.  For $p > m$, equality holds if and only if $dF_{x}$ is a homothety, by a constant factor $C_{F}$, at almost all $x \in M$ at which $F$ is differentiable.  
\end{lemma}

\begin{proof}

For $p=m$, Lemma \ref{elementary_lemma} follows from the pointwise inequality $m^{\frac{p}{2}}|\det(dF_{x})|^{\frac{p}{m}} \leq e_{p}(F)_{x}$, which follows from (\ref{energy_frame_eqn}) and the arithmetic-geometric mean inequality for the eigenvalues of $F^{*}h$ relative to $g$.  For $p > m$, Lemma \ref{elementary_lemma} follows from this pointwise inequality together with H\"older's inequality. \end{proof}

Note that in Lemma \ref{elementary_lemma}, equality for any $p > m$ implies equality for all $p \geq m$. \\ 

If $F$ is smooth, the equality condition for $p=m$ in Lemma \ref{elementary_lemma} says that $F$ is a semiconformal mapping, that is $F^{*}h = \varphi g$ for a nonnegative function $\varphi$ on $M$, and the equality condition for $p > m$ says that $F$ is a homothety, i.e. $F^{*}h$ is a rescaling of $g$.  This generalizes the well-known fact that for mappings of surfaces, the energy is pointwise bounded below by the area of the image, with equality precisely where the mapping is conformal.  The uniformization theorem implies that every Riemannian metric $g$ on $\R P^{2}$ is conformally equivalent to a constant curvature metric $g_{0}$ which is unique up to scale.  Lemma \ref{elementary_lemma} implies that for a metric $g = \varphi g_{0}$ conformal to a constant curvature metric $g_{0}$, the identity mapping from $(\R P^{2},g_{0})$ to $(\R P^{2},g)$ minimizes energy in its homotopy class and has energy equal to $2A(\R P^{2},g)$ (with the energy defined with our normalization in (\ref{energy_eqn})).  Pu's Theorem \ref{pu_thm} is then a special case of Croke's lower bound for the energy of mappings $F:(\R P^{n},g_{0}) \rightarrow (M,g)$. \\  

The following formula for the energy density is used in Croke's results in \cite{Cr1} and will also be used throughout the proofs of our results below:  

\begin{lemma}[See \cite{Cr1}]
Let $F:(M^{m},g) \rightarrow (N^{n},h)$ be a Lipschitz mapping of Riemannian manifolds and $x \in M$ a point at which $F$ is differentiable.  Then:   

\begin{equation}
\label{norm_squared_formula}
\displaystyle |dF_x|^{2} = \frac{m}{\sigma(m-1)} \int\limits_{U_{x}(M,g)} |dF(\vec{u})|^{2} d\vec{u}. 
\end{equation}
\end{lemma}

The identity in (\ref{norm_squared_formula}) is the basis for the following formula for the energy of a mapping $F:(\C P^{N},g_{0}) \rightarrow (M,g)$, where $g_{0}$ is a K\"ahler metric on $\C P^{N}$.  This result is an elementary example of the type of arguments and calculations we will employ below and is also an important lemma in several of them:  

\begin{lemma} 
\label{energy_formula_lemma}
Let $g_{0}$ be the canonical metric on $\C P^{N}$, normalized so that its sectional curvature $K$ satisfies $1 \leq K \leq 4$.  Let $\mathcal{L}(\C P^{N})$ be the family of linearly embedded $1$-dimensional complex projective subspaces $\C P^{1} \subseteq \C P^{N}$, let $U(\C P^{N},g_{0})$ be the unit tangent bundle of $(\C P^{N},g_{0})$ and let $\mathsf{T}:U(\C P^{N},g_{0}) \rightarrow \mathcal{L}(\C P^{N})$ be the mapping which sends a tangent vector $\vec{u} \in U(\C P^{N},g_{0})$ to the unique element of $\mathcal{L}(\C P^{N})$ to which $\vec{u}$ is tangent.  Let $d\widetilde{\mathcal{M}}$ be the measure on $\mathcal{L}(\C P^{N})$ which is pushed forward from the Riemannian volume on $U(\C P^{N},g_{0})$ via $\mathsf{T}$, normalized so that, if $d\mathcal{P}$ is the fibrewise volume form of the fibres of $\mathsf{T}$, $\mathsf{T}^{*}d\widetilde{\mathcal{M}} \wedge d\mathcal{P} = dVol_{U(\C P^{N},g_{0})}$.   \\ 

Let $F:(\C P^{N},g_{0}) \rightarrow (M,g)$ be a Lipschitz mapping to a Riemannian manifold $(M,g)$.  Letting $E_{2}(F|_{\mathcal{P}})$ be the energy of the mapping $F$ restricted to $\mathcal{P} \in \mathcal{L}(\C P^{N})$, we have: 

\begin{equation}
\displaystyle E_{2}(F) = \frac{2\pi N}{\sigma(2N-1)} \int\limits_{\mathcal{L}(\C P^{N})} E_{2}(F|_{\mathcal{P}}) d\widetilde{\mathcal{M}}.  
\end{equation}
\end{lemma}

\begin{proof}

By (\ref{norm_squared_formula}), we have: 

\begin{equation*}
\displaystyle E_{2}(F) = \int\limits_{\C P^{N}} |dF_{x}|^{2} dVol_{g_{0}} = \frac{2N}{\sigma(2N-1)} \int\limits_{U(\C P^{N},g_{0})} |dF(\vec{u})|^{2} d\vec{u} \bigskip 
\end{equation*}
\begin{equation*}
\displaystyle = \frac{2N}{\sigma(2N-1)} \int\limits_{\mathcal{L}(\C P^{N})} \int\limits_{U(\mathcal{P},g_{0})} |dF(\vec{u})|^{2} \ d\vec{u} \ d\widetilde{\mathcal{M}}, \bigskip 
\end{equation*}
where $U(\mathcal{P},g_{0})$ is the unit tangent bundle of a complex projective line $\mathcal{P} \cong \C P^{1}$ in the metric $g_{0}|_{\mathcal{P}}$.  Using (\ref{norm_squared_formula}) again, 

\begin{equation*}
\displaystyle \frac{2N}{\sigma(2N-1)} \int\limits_{\mathcal{L}(\C P^{N})} \int\limits_{U(\mathcal{P},g_{0})} |dF(\vec{u})|^{2} \ d\vec{u} \ d\widetilde{\mathcal{M}} = \frac{2\pi N}{\sigma(2N-1)} \int\limits_{\mathcal{L}(\C P^{N})} \int\limits_{\mathcal{P}} |dF|_{\mathcal{P}_{x}}|^{2} \ dx \ d\widetilde{\mathcal{M}} \bigskip 
\end{equation*}
\begin{equation*}
\displaystyle = \frac{2\pi N}{\sigma(2N-1)} \int\limits_{\mathcal{L}(\C P^{N})} E_{2}(F|_{\mathcal{P}}) d\widetilde{\mathcal{M}}.  
\end{equation*}
\end{proof} 

One can derive similar formulas for the energy of mappings of $(\R P^{n},g_{0})$, $(\mathbb{H} P^{N},g_{0})$ and the Cayley plane $(\mathcal{C}a P^{2},g_{0})$ as integrals over the spaces of linearly embedded $\R P^{1}$, $\mathbb{H} P^{1}$ and $\mathcal{C}a P^{1} \cong S^{8}$, and over the space of geodesics in the the sphere $(S^{n},g_{0})$.  The formula of this type for the energy of mappings $F:(\R P^{n},g_{0}) \rightarrow (M,g)$ plays a key part in Croke's proof of the $p=2$ case of Theorem \ref{rpn_thm} in \cite{Cr1}.  Even in the cases $(\mathbb{H}P^{N},g_{0})$, $(\mathcal{C}a P^{2},g_{0})$ and $(S^{n},g_{0}),n \geq 3$, where one knows that the identity mapping is homotopic to maps with arbitrarily small energy, one can use such a formula to show that a family of mappings whose energies decay to $0$ must also have energies decaying to $0$ when restricted to almost all linear subpsaces $\mathbb{H}P^{d} \subseteq \mathbb{H}P^{N}$, $\mathcal{C}a P^{1} \subseteq \mathcal{C}a P^{2}$, and almost all totally geodesic subspheres $S^{d} \subseteq S^{n}$. \\  

We end this section with a few comments about the regularity of harmonic and energy minimizing maps: \\  

Continuous, weakly harmonic maps are smooth \cite{EL2}.  In particular, any continuous map which minimizes energy in its homotopy class is smooth.  Assuming only Lipschitz regularity, mappings which realize equality for $p=2$ in  Theorems \ref{rpn_thm}, \ref{kahler_thm}, \ref{cpn_holom_thm} and \ref{cpn_thm} are therefore $C^{\infty}$.  However, unless one has established that equality holds for $p=2$, one cannot assume smoothness because for $p \neq 2$ there are $p$-energy minimizing maps which are $C^{1,\alpha}$ for $\alpha < 1$ but are not $C^{2}$.  In dimensions $3$ and greater, $p$-energy minimizing maps also need not be continuous \cite[Section 3]{EL2}.  White has shown \cite{Wh3} that $p$-energy minimizing sequences of maps of compact Riemannian manifolds converge, in an appropriate topology, to mappings which belong to a Sobolev space of mappings and have well-defined homotopy classes when restricted to lower-dimensional skeleta of their domain.  However a $p$-energy minimizing sequence of maps in one homotopy class can converge, in a weak sense, to a map in another homotopy class -- for example, on $(S^{n},g_{0})$, a family of conformal dilations with energy decaying to $0$, in the homotopy class of the identity, converges weakly to a constant map.  In light of these results, Lipschitz regularity is a stronger assumption than is natural for $p$-energy minimizing maps in general.  However, the Lipschitz condition works well in our setting because it is inherited by the restriction of maps $F:(M,g) \rightarrow (N,h)$ to any submanifold of $M$.  In our case, this implies that a Lipshitz mapping $F:(\mathbb{K}P^{N},g_{0}) \rightarrow (M,g)$ will be Lipshitz when restricted to all $\mathbb{K}P^{1} \subseteq \mathbb{K}P^{N}$, where $\mathbb{K}$ is $\R$, $\C$ or $\mathbb{H}$.  For $\mathbb{H} P^{N}$, we will also consider the restriction of $F$ to a family of totally geodesic submanifolds of $(\mathbb{H} P^{N},g_{0})$ isometric to $(\C P^{2},g_{0})$.  This will allow us to draw conclusions about the mapping of $\mathbb{K}P^{N}$ from its behavior along lower-dimensional subspaces.   


\section{Real Projective Space}  
\label{reals} 


In this section, we will prove Theorem \ref{rpn_thm}.  We define the space of oriented geodesics in $(\R P^{n},g_{0})$ to be the quotient of the unit tangent bundle $U(\R P^{n}, g_{0})$ by the geodesic flow.  We denote this $\mathcal{G}(\R P^{n})$ and we equip $\mathcal{G}(\R P^{n})$ with the measure $d\gamma$ pushed forward from the measure on $U(\R P^{n}, g_{0})$.  Because $Vol(U(\R P^{n}, g_{0})) = \frac{\sigma(n)\sigma(n-1)}{2}$, we have:   

\begin{equation}
\label{Vol_G_equation}
\displaystyle Vol(\mathcal{G}(\R P^{n})) = \frac{Vol(U(\R P^{n}, g_{0}))}{\pi}  = \frac{\sigma(n)\sigma(n-1)}{2\pi}. \bigskip 
\end{equation}

\begin{proof}[Proof of Theorem \ref{rpn_thm}]

By the formula (\ref{norm_squared_formula}) for the energy density and the Cauchy-Schwarz inequality, 

\begin{equation*}
\displaystyle E_{p}(F) = \int\limits_{\R P^{n}} |dF_{x}|^{p} dVol_{g_{0}} = \int\limits_{\R P^{n}} \left( \frac{n}{\sigma(n-1)} \int\limits_{U_{x}(\R P^{n},g_{0})} |dF(\vec{u})|^{2} d\vec{u} \right)^{(\frac{p}{2})} dVol_{g_{0}} \bigskip 
\end{equation*}
\begin{equation}
\label{rpn_pf_eqn_1}
\displaystyle \geq \frac{n^{\frac{p}{2}}}{\sigma(n-1)^{p}} \int\limits_{\R P^{n}} \left( \int\limits_{U_{x}(\R P^{n},g_{0})} |dF(\vec{u})| d\vec{u} \right)^{p} dVol_{g_{0}}. \bigskip 
\end{equation}

For $p = 1$, this says: 

\begin{equation}
\label{rpn_pf_eqn_2}
\displaystyle E_{1}(F) \geq \frac{\sqrt{n}}{\sigma(n-1)} \int\limits_{U(\R P^{n},g_{0})} |dF(\vec{u})| d\vec{u}. \bigskip 
\end{equation}

For $p > 1$, by (\ref{rpn_pf_eqn_1}) and H\"older's inequality, 

\begin{equation}
\label{rpn_pf_eqn_3}
\displaystyle E_{p}(F) \geq \frac{2^{p-1} n^{\frac{p}{2}}}{\sigma(n-1)^{p} \sigma(n)^{p-1}} \left( \int\limits_{U(\R P^{n},g_{0})} |dF(\vec{u})| d\vec{u} \right)^{p}. \bigskip 
\end{equation}

For each $\gamma \in \mathcal{G}(\R P^{n})$, $F \circ \gamma$ is an oriented Lipschitz $1$-cycle in $(M,g)$.  Letting $|F \circ \gamma|$ be its mass, and writing $\gamma: [0,\pi] \rightarrow \R P^{n}$ for a unit-speed parametrization of $\gamma$ and $F \circ \gamma: [0,\pi] \rightarrow M$ for the associated parametrization of $F \circ \gamma$, we then have:  

\begin{equation*}
\displaystyle |F \circ \gamma| = \int\limits_{0}^{\pi} |(F \circ \gamma)'(t)| dt, \bigskip 
\end{equation*}
where we have used that the Lipschitz mapping $F \circ \gamma: [0,\pi] \rightarrow M$ is differentiable almost everywhere.  For all $t$ such that $F$ is differentiable at $\gamma(t)$, $(F \circ \gamma)'(t) = dF(\gamma'(t))$.  By Fubini's theorem, the right-hand sides of (\ref{rpn_pf_eqn_2}) and (\ref{rpn_pf_eqn_3}) can therefore be rewritten in terms of integrals over $\mathcal{G}(\R P^{n})$, which implies: 

\begin{equation}
\displaystyle E_{p}(F) \geq \frac{2^{p-1} n^{\frac{p}{2}}}{\sigma(n-1)^{p} \sigma(n)^{p-1}} \left( \int\limits_{\mathcal{G}(\R P^{n})} |F \circ \gamma| d\gamma \right)^{p}. \bigskip 
\end{equation}

Because each geodesic $\gamma$ represents a generator of $\pi_{1}(\R P^{n})$, $|F \circ \gamma| \geq L^{\star}$, and therefore, 

\begin{equation*}
\displaystyle E_{p}(F) \geq \frac{2^{p-1} n^{\frac{p}{2}}}{\sigma(n-1)^{p} \sigma(n)^{p-1}} \left( Vol(\mathcal{G}(\R P^{n})) L^{\star} \right)^{p}, \bigskip  
\end{equation*}
which is (\ref{rpn_thm_eqn}). \\  

Suppose equality holds for $p=1$. \\ 

This implies that equality holds in the Cauchy-Schwarz inequality in (\ref{rpn_pf_eqn_1}) for a.e. $x \in \R P^{n}$.  For all $x$ at which $F$ is differentiable and for which this equality holds, $|dF_{x}(\vec{u})|$ depends only on $x$.  This implies that $F^{*}g$ is a.e. equal to $\varphi(x) g_{0}$, where $\varphi(x)$ is a nonnegative function on $\R P^{n}$.  Because all $\gamma \in \mathcal{G}(\R P^{n})$ map to rectifiable currents $F \circ \gamma$ with well-defined lengths in $(M,g)$, equality also implies that for almost all $\gamma$, $|F \circ \gamma| = L^{\star}$.  Because $|F \circ \gamma| \geq L^{\star}$ and $|F \circ \gamma|$ is lower semicontinuous on $\mathcal{G}(\R P^{n})$, we in fact have $|F \circ \gamma| = L^{\star}$ for all $\gamma$.  The image via $F$ of each geodesic $\gamma$ is therefore a closed geodesic in $(M,g)$, of minimal length $L^{\star}$ in its free homotopy class, although a priori $F \circ \gamma$ may not be parametrized by arc length. \\ 

If $F$ is a smooth immersion, then because each geodesic $\gamma$ in $(\R P^{n},g_{0})$ maps to a closed geodesic in $M$, the image of $F$ is a totally geodesic submanifold, and because $F \circ \gamma$ is of minimal length in its free homotopy class in $(M,g)$, $F^{*}g$ is a Blaschke metric on $\R P^{n}$, cf. Remark \ref{blaschke_remark} below.  By the Berger-Green-Kazdan-Yang proof of the Blaschke conjecture \cite{Gn, Be1}, $F^{*}g$ is therefore isometric to a round metric.  This does not yet imply that $F$ is an isometry or a homothety.  However, letting $\psi: (\R P^{n}, F^{*}g) \rightarrow (\R P^{n}, g_{0})$ be an isometry (or homothety), we then have $\psi^{*}g_{0} = F^{*}g = \varphi(x) g_{0}$, where $\varphi(x)$ is the semiconformal factor as above and is in fact a conformal factor, i.e. is everywhere-defined and positive, because $F$ is an immersion.  By the classification of conformal diffeomorphisms of the round sphere $(S^{n}, g_{0})$, the mapping $\psi$ is therefore an isometry of $g_{0}$, up to rescaling, and $F$ is an isometry or homothety onto its image. \\

Now suppose $p > 1$ and equality holds for $p$. \\ 

Supposing only that $F$ is Lipschitz, this implies all of the conditions which follow for Lipshitz mappings which realize equality for $p=1$ and also implies equality in H\"older's inequality in (\ref{rpn_pf_eqn_3}).  Equality in H\"older's inequality implies that the semiconformal factor $\varphi(x)$ is a.e. equal to a constant $C_{F}$. This implies that equality holds for all $p \geq 1$.  Because equality holds for $p = 2$, $F$ is smooth, so $\varphi(x)$ is an everywhere-defined constant function and $F$ is a homothety onto its image.  Because all $\gamma \in \mathcal{G}(\R P^{n},g_{0})$ map to geodesics of $(M,g)$, the image of $F$ is a totally geodesic submanifold. \end{proof}

\begin{remark}
\label{blaschke_remark}
To see that the conditions for equality when $p=1$ and $F$ is an immersion imply that $F^{*}g$ is a Blaschke metric, i.e. that the first conjugate locus of each point $x_{0}$ in $(\R P^{N},F^{*}g)$ is a single point, in fact $x_{0}$, note that each unit-speed geodesic $c:[0,L^{\star}] \rightarrow (\R P^{N},F^{*}g)$ has a conjugate point at $c(L^{\star}) = c(0)$, where all geodesics based at $c(0)$ intersect, and that this must be the first conjugate point to $c(0)$ along $c$ because $c([0,L^{\star}])$ is length minimizing in its homotopy class. 
\end{remark}

\begin{corollary} 
\label{rpn_cor}
The identity mapping of $(\R P^{n},g_{0})$, $n \geq 2$, minimizes $p$-energy in its homotopy class for all $p \geq 1$.  Up to isometries, the identity mapping is the unique $p$-energy minimizing map in its homotopy class for $p > 1$ and the unique $p$-energy minimizing immersion in its homotopy class for $p=1$. 
\end{corollary}
  
Note that the characterization of equality for $p = 1$ in Theorem \ref{rpn_thm} and Corollary \ref{rpn_cor} is false without the stipulation that $n \geq 2$:  any diffeomorphism of $\R P^{1} = \R / \pi\Z$ has $1$-energy equal to $\pi$. \\ 

We note that for $n \geq 2$, non-isometric projective linear transformations of $\R P^{n}$ have the length-preserving property for $\gamma \in \mathcal{G}(\R P^{n})$ which is implied by equality for $p=1$ in Theorem \ref{rpn_thm} but do not minimize $1$-energy in their homotopy class.  In complex projective space, however, non-isometric projective complex linear transformations are holomorphic and minimize $2$-energy in the homotopy class of the identity. \\  

We end this section by re-interpreting our proof that the identity mapping of $(\R P^{n},g_{0})$ minimizes $1$-energy in its homotopy class in terms of a calibration-like property of the canonical $1$-form on the unit tangent bundle of $(\R P^{n},g_{0})$: \\  

For any oriented integral $1$-chain $\sum_{i} a_{i}\tau_{i}$ in $\R P^{n}$, where $a_{i} \in \Z$ and $\tau_{i}: \Delta^{1} \rightarrow \R P^{n}$ are oriented Lipschitz $1$-simplices, one can define a $1$-current in the unit tangent bundle $U(\R P^{n},g_{0})$ by associating to each point $\tau_{i}(t)$ at which $\tau_{i}'(t) \neq 0$ the unit vector in $T_{\tau_{i}(t)}\R P^{N}$ tangent to $\tau_{i}$ in the oriented direction, with multiplicity $a_{i}$.  Letting $\alpha$ be the canonical $1$-form on the unit tangent bundle of $(\R P^{n},g_{0})$, when all $a_{i}$ are non-negative, the mass of this current is its pairing with $\alpha$.  The current of this type associated to any non-contractible closed curve in $\R P^{n}$ has mass greater than or equal to $\pi$, and any Lipshitz mapping $F:(\R P^{n},g_{0}) \rightarrow (\R P^{n},g_{0})$ homotopic to the identity sends each $\gamma \in \mathcal{G}(\R P^{n})$ to such a closed curve.  This implies the lower bound for the $1$-energy of mappings homotopic to $Id:(\R P^{n},g_{0}) \rightarrow (\R P^{n}, g_{0})$ in Theorem \ref{rpn_thm}.  Unlike a calibration, however, $\alpha$ is not closed:  $d\alpha$ is the Liouville symplectic form on the tangent bundle. 


\section{Complex Projective Space}
\label{complexes}


In this section, we will prove lower bounds for the $p$-energy of mappings from $\C P^{N}$ to Riemannian manifolds, similar to the estimates for mappings of real projective space in Theorem \ref{rpn_thm}.  The equality case in these results is much broader than in Theorem \ref{rpn_thm}, but for mappings of $\C P^{2}$ it implies Theorem \ref{kahler_thm}.  For mappings from $\C P^{N}$ to simply connected, compact K\"ahler manifolds, we will prove the stronger characterization of equality in Theorem \ref{cpn_holom_thm}. \\ 

As in Lemma \ref{energy_formula_lemma}, we will let $\mathcal{L}(\C P^{N})$ be the space of linearly embedded $\C P^{1} \subseteq \C P^{N}$ and $d\widetilde{\mathcal{M}}$ the measure on $\mathcal{L}(\C P^{N})$ pushed forward from the measure on the unit tangent bundle $U(\C P^{N}, g_{0})$ via the mapping $\mathsf{T}:U(\C P^{N}, g_{0}) \rightarrow \mathcal{L}(\C P^{N})$.  The total volume of $\mathcal{L}(\C P^{N})$ via the measure $d\widetilde{\mathcal{M}}$ is then equal to:  

\begin{equation*}
\displaystyle Vol(\mathcal{L}(\C P^{N}),d\widetilde{\mathcal{M}}) = \frac{Vol(U(\C P^{N},g_{0}))}{Vol(U(\C P^{1}, g_{0}))} = \frac{1}{2 \pi^{2}} \times \frac{\sigma(2N-1)\pi^{N}}{N!}
\end{equation*}
\begin{equation}
\displaystyle = \frac{\sigma(2N-1)\pi^{N-2}}{2N!}. \bigskip 
\end{equation}

It is a fundamental fact of K\"ahler geometry that closed complex submanifolds of a compact K\"ahler manifold $(X,h)$ are area minimizing in their homology classes -- this follows from the calibrated structure given by the K\"ahler form, cf. \cite{HL1}.  Moreover, for any compact complex curve $\Sigma$ and any holomorphic mapping $F$ of $\Sigma$ to a compact K\"ahler manifold $(X,h)$, it follows from Lemma \ref{elementary_lemma} that: 

\begin{equation}
\displaystyle E_{2}(F) = 2A(\Sigma,F^{*}h), \bigskip 
\end{equation}
where $A(\Sigma,F^{*}h)$ is equal to the area of $F(\Sigma)$ and is the minimum area of any cycle representing the class of $F(\Sigma)$ in $H_{2}(X;\Z)$.  For any holomorphic mapping $F:(\C P^{N},g_{0}) \rightarrow (X,h)$, letting $A^{\star}$ be this minimum area for $\mathcal{P} \in \mathcal{L}(\C P^{N})$, Lemma \ref{energy_formula_lemma} then implies: 

\begin{equation*}
\displaystyle E_{2}(F) = \frac{2\pi N}{\sigma(2N-1)} \int\limits_{\mathcal{L}(\C P^{N})} E_{2}(F|_{\mathcal{P}}) d\widetilde{\mathcal{M}} 
\end{equation*}
\begin{equation}
\displaystyle = \frac{2\pi N}{\sigma(2N-1)} \times Vol(\mathcal{L}(\C P^{N}),d\widetilde{\mathcal{M}}) \times 2A^{\star} = \frac{2 \pi^{N-1}}{(N-1)!} A^{\star}. \bigskip
\end{equation}

We will see that this coincides with a sharp lower bound for the energy of any Lipschitz mapping from $(\C P^{N},g_{0})$ to a Riemannian manifold $(M,g)$:

\begin{theorem}
\label{cpn_thm}
Let $(\C P^{N},g_{0})$ be complex projective space with its canonical metric $g_{0}$ with sectional curvature $K$ satisfying $1 \leq K \leq 4$.  Let $F: (\C P^{N}, g_{0}) \rightarrow (M^{m},g)$ be a Lipschitz mapping to a Riemannian manifold $(M,g)$ and $A^{\star}$ the infimum of the areas of Lipschitz mappings $f:S^{2} \rightarrow (M,g)$ in the free homotopy class of $F_{*}(\C P^{1})$. \\ 

Then for all $p \geq 2$, 

\begin{equation}
\label{cpn_thm_eqn}
\displaystyle E_{p}(F) \geq \frac{\pi^{N}}{N!} \left( \frac{2N}{\pi} A^{\star} \right)^{\frac{p}{2}}. \medskip 
\end{equation}

Suppose equality holds for $p = 2$.  Then $F^{*}g$ is a positive semidefinite Hermitian bilinear form on $\C P^{N}$.  In particular, on the domain $\mathcal{V} \subseteq \C P^{N}$ on which $rk(dF) = 2N$, $F^{*}g$ is a Hermitian metric.  Letting $\omega^{*}$ denote the K\"ahler form of $F^{*}g$ on $\mathcal{V}$, for all $x \in \mathcal{V}$ and all $k = 1, 2, \dots, N-1$, $d(\omega^{*^{k}})$ vanishes on all complex subspaces of $T_{x}\mathcal{V}$ of complex dimension $k+1$.  In particular, for all complex surfaces $Y$ in $\C P^{N}$, $F^{*}g|_{Y \cap \mathcal{V}}$ is a K\"ahler metric, and $\omega^{*^{N-1}}$ is closed on $\mathcal{V}$.  $F(\mathcal{V})$ is a minimal submanifold of $(M,g)$, and the second fundamental form of $F(\mathcal{V})$ in $(M,g)$ can be diagonalized by a unitary basis of $F^{*}g$.  Equality for $p=2$ for $F:\C P^{N} \rightarrow (M,g)$ implies equality for $p=2$ for $F|_{\C P^{d}}$ for all linear subspaces $\C P^{d} \subseteq \C P^{N}$. \\ 

If $p > 2$ and equality holds for $p$, then equality holds for all $p  \geq 2$ and $F$ has constant energy density.  If $p > 2$, $\C P^{2} \subseteq \C P^{N}$ is a linearly embedded subspace and $F|_{\C P^{2}}$ is an immersion and realizes equality for $p$, then $F|_{\C P^{2}}$ is a homothety onto its image. 
\end{theorem}

Before proving Theorem \ref{cpn_thm} we will prove: 

\begin{lemma}
\label{kahler_form_prop}
Let $(X,h)$ be a Hermitian manifold of complex dimension $N$ and $\omega$ the K\"ahler form of the metric $h$.  Suppose that for all $x_{0} \in X$ and all complex lines $\Pi \subseteq T_{x_{0}}X$ there is a complex curve $\Sigma_{\Pi} \subseteq X$, with $\Pi$ tangent to $\Sigma_{\Pi}$, whose mean curvature vanishes at $x_{0}$. \\  

Then for all $k = 1, 2, \dots, N-1$, the exterior derivative of $\omega^{k}$ vanishes on all complex subspaces of complex dimension $k+1$.  In particular, $\omega^{N-1}$ is closed. 
\end{lemma}

\begin{proof}[Proof of Lemma \ref{kahler_form_prop}] Let $I$ denote the complex structure of $X$.  We will first verify the lemma when $k=1$:  let $\vec{w}_{1}, \vec{w}_{2} = I(\vec{w}_{1})$ be a unitary basis for a complex line $\Pi$ in a tangent space $T_{x_{0}}X$.  Let $\vec{v}$ be a unit vector orthogonal to $\Pi$ and $\Sigma$ a complex curve to which $\Pi$ is tangent. \\ 

Define normal coordinates on a neighborhood of $\Sigma$ about $x_{0}$ based on the frame $\vec{w}_{1}, \vec{w}_{2}$ fpr $T_{x_{0}}\Sigma$.  Extend $\vec{v}$ to an orthonormal frame for the normal space to $\Sigma$ at $x_{0}$, extend this frame to an orthonormal frame field for the normal bundle to $\Sigma$ over the normal coordinate neighborhood defined above and use these to define Fermi coordinates on a neighborhood of $x_{0}$ in $X$.  Let $W_{1}, W_{2}, V$ be the coordinate vector fields which coincide with $\vec{w}_{1}, \vec{w}_{2}, \vec{v}$ at $x_{0}$.  Then we have: 

\begin{equation*}
\displaystyle d\omega(\vec{v}, \vec{w}_{1}, \vec{w}_{2}) = V(\omega(W_{1}, W_{2})) - W_{1}(\omega(V,W_{2})) + W_{2}(\omega(V,W_{2})) \bigskip 
\end{equation*}
\begin{equation}
\displaystyle = V(h(I(W_{1}), W_{2})) - W_{1}(h(I(V),W_{2})) + W_{2}(h(I(V),W_{2})). \bigskip 
\end{equation}

Because $W_{1},W_{2}$ are tangent and $V$ is normal to to the complex submanifold $\Sigma$ of $X$, the terms $W_{1}(\omega(V,W_{2})) = W_{1}(h(I(V),W_{2}))$ and $W_{2}(\omega(V,W_{2})) = W_{2}(h(I(V),W_{2}))$ vanish, and we have: 

\begin{equation}
\label{kfpe1}
\displaystyle d\omega(\vec{v}, \vec{w}_{1}, \vec{w}_{2}) = V(\omega(W_{1}, W_{2})) = V(h(I(W_{1}), W_{2})). \bigskip 
\end{equation}

Although $I(W_{1})$ may not be equal to $W_{2}$ at $x \neq x_{0}$, the Hermitian property of $h$ implies that $h(I(W_{1}),I(W_{1})) \equiv h(W_{1},W_{1})$.  Together with (\ref{kfpe1}), this implies that at $x_{0}$, 

\begin{equation*}
\displaystyle d\omega(\vec{v}, \vec{w}_{1}, \vec{w}_{2}) = h(\nabla_{V} I(W_{1}),I(W_{1})) + h(W_{2},\nabla_{V} W_{2}) \bigskip 
\end{equation*}
\begin{equation*}
\displaystyle  = \frac{1}{2} V \left( h(I(W_{1}),I(W_{1})) + h(W_{2},W_{2}) \right) \bigskip 
\end{equation*}
\begin{equation}
\label{kfpe2}
\displaystyle = \frac{1}{2} V \left( h(W_{1},W_{1}) + h(W_{2},W_{2}) \right). \bigskip 
\end{equation}
Because $V,W_{1},W_{1}$ are coordinate vector fields from the same Fermi coordinate system defined above, $\nabla_{V}W_{1} - \nabla_{W_{1}}V = [V,W_{1}] = 0$ and $\nabla_{V}W_{2} - \nabla_{W_{2}}V = [V,W_{2}] = 0$.  By (\ref{kfpe2}), we then have: 

\begin{equation*}
\displaystyle d\omega(\vec{v}, \vec{w}_{1}, \vec{w}_{2}) = h(\nabla_{V} W_{1},W_{1}) + h(W_{2},\nabla_{V} W_{2}) \bigskip 
\end{equation*}
\begin{equation}
\displaystyle = h(\nabla_{W_{1}} V,W_{1}) + h(W_{2},\nabla_{W_{2}} V). \bigskip 
\end{equation}

This is the negative of the mean curvature of $\Sigma$ at $x_{0}$ in the normal direction determined by $\vec{v}$, which is $0$ by assumption.  Because $d\omega$ vanishes on all triples of the form $\vec{v},\vec{w}_{1},\vec{w}_{2} = I(\vec{w}_{1})$ as above, it vanishes on all complex subspaces of complex dimension $2$.  This establishes Lemma \ref{kahler_form_prop} when $k = 1$. \\  

For $2 \leq k \leq N-1$, let $\vec{w}_{1},\vec{w}_{2} = I(\vec{w}_{1}),\dots,\vec{w}_{2k-1},\vec{w}_{2k} = I(\vec{w}_{2k-1})$ be a unitary frame and $\vec{v}$ a unit vector orthogonal to $span(\vec{w}_{1},\dots,\vec{w}_{2k})$.  We then have:  

\begin{equation}
\label{kfpe3}
\displaystyle d(\omega^{k})(\vec{v}, \vec{w}_{1},\dots,\vec{w}_{2k}) = k d\omega \wedge \omega^{k-1}(\vec{v}, \vec{w}_{1},\dots,\vec{w}_{2k}). \bigskip 
\end{equation}

Letting $\vec{w}_{0} = \vec{v}$ for notational convenience, and letting $\mathcal{S}$ denote the permutations of $0,1,2,\dots,2k-1,2k$ with $\sigma(0) < \sigma(1) < \sigma(2)$ and $\sigma(3) < \cdots < \sigma(2k)$, (\ref{kfpe3}) implies that $d(\omega^{k})(\vec{v}, \vec{w}_{1},\dots,\vec{w}_{2k})$ is equal to: 

\begin{equation}
\displaystyle k \sum\limits_{\sigma \in \mathcal{S}} sgn(\sigma) d\omega(\vec{w}_{\sigma(0)},\vec{w}_{\sigma(1)},\vec{w}_{\sigma(2)})\omega^{k-1}(\vec{w}_{\sigma(3)},\dots,\vec{w}_{\sigma(2k)}). \bigskip 
\end{equation}

Unless $\left(\sigma(3),\dots,\sigma(2k) \right) = \left(1,2,\dots,2j,2j+3,\dots,2k \right)$ for some $j = 1, 2, \dots k$, we have $\omega^{k-1}(\vec{w}_{\sigma(3)},\dots,\vec{w}_{\sigma(2k)}) = 0$.  In that case, $\left(\sigma(0),\sigma(1),\sigma(2) \right) = \left(0,2j+1,2j+2 \right)$ and $d\omega(\vec{w}_{\sigma(0)},\vec{w}_{\sigma(1)},\vec{w}_{\sigma(2)}) = 0$ by the $k = 1$ case.  Because $d(\omega^{k})$ vanishes on all $(2k+1)$-tuples $\vec{v},\vec{w}_{1},\dots,\vec{w}_{2k}$ as above, it vanishes on all complex subspaces of complex dimension $k+1$.  \end{proof}

\begin{proof}[Proof of Theorem \ref{cpn_thm}] Let $I$ denote the complex structure of $\C P^{N}$.  By (\ref{norm_squared_formula}), 

\begin{equation*}
\displaystyle E_{p}(F) = \int\limits_{\C P^{N}} |dF_{x}|^{p} dVol_{g_{0}} 
\end{equation*}
\begin{equation}
\label{cpn_pf_eqn_1}
\displaystyle = \left( \frac{2N}{\sigma(2N-1)} \right)^{\frac{p}{2}} \int\limits_{\C P^{N}} \left( \int\limits_{U_{x}(\C P^{N},g_{0})} |dF(\vec{u})|^{2} d\vec{u} \right)^{\frac{p}{2}} dVol_{g_{0}}. \medskip 
\end{equation}

For $x \in \C P^{N}$, let $G_{1}^{\C}(x)$ be the space of complex lines in $T_{x}\C P^{N}$.  $G_{1}^{\C}(x)$ is then a copy of $\C P^{N-1}$, and the fibration $\mathsf{T}:U(\C P^{N}, g_{0}) \rightarrow \mathcal{L}(\C P^{N})$ factors through a quotient mapping which is given fibrewise by $U_{x}(\C P^{N},g_{0}) \rightarrow G_{1}^{\C}(x)$.  For $\Pi \in G_{1}^{\C}(x)$, let $U_{\Pi}$ be the unit circle in $\Pi$ in the metric $g_{0}$.  For each such $\Pi$ at each point $x$ where $F$ is differentiable, $F^{*}g|_{\Pi}$ is a positive semidefinite, symmetric $2$-form and can be diagonalized relative to the metric $g_{0}|_{\Pi}$.  Letting $\vec{u}_{1}, \vec{u}_{2}$ be eigenvectors for $F^{*}g|_{\Pi}$ which are orthonormal in $g_{0}$, we have: 

\begin{equation*}
\displaystyle \int\limits_{U_{\Pi}} |dF(\vec{u})|^{2} d\vec{u} = \int\limits_{0}^{2\pi} \left( \cos^{2}(\theta) |dF(\vec{u}_{1})|^{2} + \sin^{2}(\theta) |dF(\vec{u}_{2})|^{2} \right) d\theta \medskip  
\end{equation*}   
\begin{equation}
\label{cpn_pf_eqn_2}
\displaystyle = \pi \left( |dF(\vec{u}_{1})|^{2} + |dF(\vec{u}_{2})|^{2} \right) \geq 2\pi |dF(\vec{u}_{1})||dF(\vec{u}_{2})| = 2\pi |\det(dF|_{\Pi})|. \bigskip 
\end{equation}

For $p = 2$, (\ref{cpn_pf_eqn_1}), (\ref{cpn_pf_eqn_2}) and Fubini's theorem imply: 

\begin{equation*}
\displaystyle E_{2}(F) = \frac{2N}{\sigma(2N-1)} \int\limits_{\C P^{N}} \int\limits_{G_{1}^{\C}(x)} \int\limits_{U_{\Pi}} |dF(\vec{u})|^{2} \ d\vec{u} \ d\Pi \ dVol_{g_{0}} \bigskip 
\end{equation*}
\begin{equation}
\label{cpn_pf_eqn_3}
\displaystyle \geq \frac{4N\pi}{\sigma(2N-1)} \int\limits_{\C P^{N}} \int\limits_{G_{1}^{\C}(x)} |\det(dF|_{\Pi})| d\Pi dVol_{g_{0}} = \frac{4N\pi}{\sigma(2N-1)} \int\limits_{\mathcal{L}(\C P^{N})} |F(\mathcal{P})| d\widetilde{\mathcal{M}}, \bigskip 
\end{equation}
where $|F(\mathcal{P})|$ is the area of the image via $F$ of a complex projective line $\mathcal{P}$ in $\C P^{N}$ and where the measure on $G_{1}^{\C}(x)$ is its canonical measure as a quotient of the unit sphere $U_{x}(\C P^{N},g_{0})$.  As in the proof of Theorem \ref{rpn_thm}, we have used the fact that $F|_{\mathcal{P}}$ is Lipschitz, so that $F(\mathcal{P})$ represents a closed integral $2$-current with a well-defined mass $|F(\mathcal{P})|$, for all $\mathcal{P} \in \mathcal{L}(\C P^{N})$. \\  

For $p > 2$, (\ref{cpn_pf_eqn_1}), (\ref{cpn_pf_eqn_2}), Fubini's theorem and H\"older's inequality imply:  

\begin{equation}
\label{cpn_pf_eqn_4}
\displaystyle E_{p}(F) \geq \frac{(4N\pi)^{\frac{p}{2}}}{\sigma(2N-1)^{\frac{p}{2}}Vol(\C P^{N})^{\frac{p-2}{2}}} \left( \int\limits_{\mathcal{L}(\C P^{N})} |F(\mathcal{P})| d\widetilde{\mathcal{M}} \right)^{\frac{p}{2}}. \medskip 
\end{equation}

Since $|F(\mathcal{P})| \geq A^{\star}$, (\ref{cpn_pf_eqn_3}) and (\ref{cpn_pf_eqn_4}) imply: 

\begin{equation}
\label{cpn_pf_eqn_5}
\displaystyle E_{p}(F) \geq \frac{(4N\pi)^{\frac{p}{2}}}{\sigma(2N-1)^{\frac{p}{2}}Vol(\C P^{N})^{\frac{p-2}{2}}} \left( A^{\star} Vol(\mathcal{L}(\C P^{N}),d\widetilde{M}) \right)^{\frac{p}{2}}, \medskip 
\end{equation}
which is (\ref{cpn_thm_eqn}). \\ 

Suppose equality holds for $p=2$. \\ 

$F$ is therefore smooth.  Equality holds in (\ref{cpn_pf_eqn_2}) for all $\Pi \in G_{1}^{\C}(x)$, at all $x \in \C P^{N}$.  Equality in (\ref{cpn_pf_eqn_2}) implies that $|dF(\vec{u})|$ is $U(1)$-invariant on $\Pi$.  $F^{*}g$ is therefore a $U(1)$-invariant, positive semidefinite bilinear form on $T_{x}\C P^{N}$.  In particular, $F^{*}g$ is a Hermitian metric on $\mathcal{V}$.  We also have that $|F(\mathcal{P})| = A^{\star}$ for all $\mathcal{P} \in \mathcal{L}(\C P^{N})$ and, because $F|_{\mathcal{P}}$ is conformal, that $F|_{\mathcal{P}}$ minimizes energy in its homotopy class of mappings $F:\C P^{1} \rightarrow (M,g)$ for all $\mathcal{P} \in \mathcal{L}$.  By Lemma \ref{energy_formula_lemma}, equality then holds for $F:(\C P^{d},g_{0}) \rightarrow (M,g)$ for all linear subspaces $\C P^{d} \subseteq \C P^{N}$. \\ 

To estabslish the minimality of $F(\mathcal{V})$ and the properties of its second fundamental form, let $x_{0} \in \mathcal{V}$ and $\vec{n}$ a unit normal vector to $F(\C P^{N})$ in $(M,g)$ at $F(x_{0})$.  Let $S_{\vec{n}}^{\C P^{N}}$ be the shape operator of $F(\C P^{N})$ in the normal direction $\vec{n}$.  Let $\vec{u}_{0}$ be a principal vector for $S_{\vec{n}}^{\C P^{N}}$, with $|S_{\vec{n}}^{\C P^{N}}(\vec{u}_{0})|$ maximal, and let $\mathcal{P}_{0} = \mathsf{T}(\vec{u}_{0})$.  Let $S_{\vec{n}}^{\mathcal{P}_{0}}$ be the shape operator of $\mathcal{P}_{0}$ in the normal direction $\vec{n}$.  Then $\vec{u}_{0}$ is also a principal vector for $S_{\vec{n}}^{\mathcal{P}_{0}}$.  $F(\mathcal{P}_{0} \cap \mathcal{V})$ is minimal in $(M,g)$ because $F(\mathcal{P}_{0})$ minimizes area in its homotopy class of mappings $\C P^{1} \rightarrow (M,g)$.  By the minimality of $F(\mathcal{P}_{0} \cap \mathcal{V})$ and the fact that $F^{*}g$ is Hermitian, $I(\vec{u}_{0})$ is also a principal vector for $S_{\vec{n}}^{\mathcal{P}_{0}}$ and $g\left(S_{\vec{n}}^{\mathcal{P}_{0}}(I(\vec{u}_{0})),I(\vec{u}_{0})\right) = -g \left(S_{\vec{n}}^{\mathcal{P}_{0}}(\vec{u}_{0}),\vec{u}_{0}\right)$.  In particular, $|S_{\vec{n}}^{\mathcal{P}_{0}}(I(\vec{u}_{0}))| = |S_{\vec{n}}^{\mathcal{P}_{0}}(\vec{u}_{0})|$.  Because $|S_{\vec{n}}^{\C P^{N}}(I(\vec{u}_{0}))| \geq |S_{\vec{n}}^{\mathcal{P}_{0}}(I(\vec{u}_{0}))| = |S_{\vec{n}}^{\mathcal{P}_{0}}(\vec{u}_{0})|$ and $|S_{\vec{n}}^{\C P^{N}}(\vec{u}_{0})|$ is maximal, this implies that $I(\vec{u}_{0})$ is also a principal vector for $S_{\vec{n}}^{\C P^{N}}$, and that the principal curvature of $S_{\vec{n}}^{\C P^{N}}$ along $I(\vec{u}_{0})$ is the negative of its principal curvature along $\vec{u}_{0}$. \\ 

Now let $\vec{u}_{1}$ be a principal vector for $S_{\vec{n}}^{\C P^{N}}$ which maximizes $|S_{\vec{n}}^{\C P^{N}}(\vec{u}_{1})|$ in the subspace of $T_{x_{0}}\C P^{N}$ which is orthogonal to $span(\vec{u}_{0},I(\vec{u}_{0}))$ in the metric $F^{*}g$, and let $\mathcal{P}_{1} = \mathsf{T}(\vec{u}_{1})$.  As above, $\vec{u}_{1}$ is a principal vector for $S_{\vec{n}}^{\mathcal{P}_{1}}$, which implies that $I(\vec{u}_{1})$ is a principal vector for $S_{\vec{n}}^{\mathcal{P}_{1}}$ whose principal curvature is the negative of the principal curvature of $\vec{u}_{1}$, and therefore that $\vec{u},I(\vec{u})$ are principal vectors of $S_{\vec{n}}^{\C P^{N}}$ whose principal curvatures are equal in magnitude and opposite in sign.  Continuing in this way, we construct a unitary basis for $T_{x_{0}}\C P^{N}$ which diagonalizes $S_{\vec{n}}^{\C P^{N}} $and shows that the mean curvature of $F(\mathcal{V})$ at $x_{0}$ in the normal direction $\vec{n}$ is $0$.  This implies that $F(V)$ is minimal in $(M,g)$.  Because $\mathcal{P} \cap \mathcal{V}$ is minimal in $(\mathcal{V},F^{*}g)$ for all $\mathcal{P} \in \mathcal{L}(\C P^{N})$, Lemma \ref{kahler_form_prop} implies that the exterior derivative of $\omega^{*^{k}}$ vanishes on all complex subspaces tangent to $\mathcal{V}$ of complex dimension $k+1$. \\ 

Now suppose $p > 2$ and equality holds for $p$. \\ 

Assuming only that $F$ is Lipschitz, equality then holds in (\ref{cpn_pf_eqn_2}) for almost all $\Pi \in G_{1}^{\C}(x)$, at almost all $x \in \C P^{N}$.  This implies that equality holds for $p=2$, and therefore that $F$ is smooth.  Equality must also hold in H\"older's inequality in (\ref{cpn_pf_eqn_4}).  This implies that $\int_{U_{x}(\C P^{N},g_{0})} |dF(\vec{u})|^{2} d\vec{u}$ is a constant function of $F$, which implies that $F$ has constant energy density. \\ 

Suppose $\C P^{2} \subseteq \C P^{N}$ is a linear subspace, $p > 2$, $F|_{\C P^{2}}$ is an immersion and realizes equality in (\ref{cpn_thm_eqn}) for $p$, and therefore for all $p \in [2,\infty)$.  The equality for $p = 2$ implies that $F^{*}g$ is a K\"ahler metric on $\C P^{2}$.  Letting $\omega^{*}$ be the K\"ahler form of $F^{*}g$ as above, for all $\C P^{1} \subseteq \C P^{2}$, $\int_{\C P^{1}}\omega^{*} = A^{\star}$.  Letting $\widetilde{\omega}$ be the K\"ahler form of $g_{0}$ on $\C P^{2}$, this implies that $\omega^{*}$ is cohomologous to $(\frac{A^{\star}}{\pi})\widetilde{\omega}$, and therefore that $Vol(\C P^{2},F^{*}g) = \frac{A^{\star^{2}}}{2}$.  For $p >4$, the lower bound in Theorem \ref{cpn_thm} therefore coincides with the lower bound in Lemma \ref{elementary_lemma}.  Because equality in (\ref{cpn_thm_eqn}) holds for all $p \geq 2$, including $p > 4$, $F$ realizes equality in Lemma \ref{elementary_lemma} for $p > 4$ and is therefore a homothety onto its image. \end{proof}

Ohnita \cite{Oh1} has proven that if $\phi:(\C P^{N},g_{0}) \rightarrow (M,g)$ is a stable harmonic map from $\C P^{N}$ with its canonical metric $g_{0}$ to any Riemannian manifold $(M,g)$, then $\phi$ is pluriharmonic.  For continuous maps, this condition is equivalent to the statement that $\phi|_{\Sigma}$ is harmonic for all complex curves $\Sigma \subseteq \C P^{N}$.  The conclusion of Theorem \ref{kahler_thm} that $F^{*}g$ is a K\"ahler metric on complex surfaces $Y \subseteq \C P^{N}$ is similar to a result of Burns, Burstall, de Bartolomeis and Rawnsley \cite[Theorem 3]{BBdBR1} that if $\phi:(M^{4},g) \rightarrow (\mathcal{Z},h)$ is stable harmonic map from a closed, real-analytic Riemannian $4$-manifold $(M^{4},g)$ to a Hermitian symmetric space $(\mathcal{Z},h)$ and there is a point of $M$ at which the rank of $d\phi$ is at least $3$, then there is a unique K\"ahler structure on $M$ with respect to which $\phi$ is holomorphic. \\ 

For $\C P^{1}$, the lower bound in Theorem \ref{cpn_thm} coincides with the bound implied by Lemma \ref{elementary_lemma}, so equality for $p > 2$ for $F:(\C P^{1},g_{0}) \rightarrow (M,g)$ in Theorem \ref{cpn_thm} implies that $F$ is a homothety onto its image.  For mappings of $(\C P^{N},g_{0})$, $N \geq 2$, let $Vol_{g_{0}}(\C P^{N}, F^{*}g)$ be the invariant of the mapping $F:(\C P^{N},g_{0}) \rightarrow (M,g)$ defined in Lemma \ref{elementary_lemma}.  We can then ask whether $Vol_{g_{0}}(\C P^{N}, F^{*}g)$ satisfies the following inequality along the lines of Pu's and Gromov's inequalities in Theorems \ref{pu_thm} and \ref{gromov_thm}: 

\begin{equation}  
\label{pu_gromov_type_inequality}
\displaystyle Vol_{g_{0}}(\C P^{N}, F^{*}g) \geq \frac{A^{\star^{N}}}{N!}. \bigskip 
\end{equation} 

For mappings for which (\ref{pu_gromov_type_inequality}) holds, the lower bound for $E_{p}(F)$ in Lemma \ref{elementary_lemma} is bounded below by the lower bound in Theorem \ref{cpn_thm}.  Mappings which realize equality in Theorem \ref{cpn_thm} for $p > 2$ and satisfy (\ref{pu_gromov_type_inequality}) must therefore be homotheties.  However, it is known that one cannot replace the limit (\ref{stable_2_systole}) in Theorem \ref{gromov_thm} by the minimum area of a current representing a generator of $H_{2}(\C P^{N};\Z)$, so (\ref{pu_gromov_type_inequality}) need not hold in general.  In fact, there are Riemannian metrics on $\C P^{2}$ of arbitrarily small volume for which the minimum area of a cycle generating $H_{2}(\C P^{N};\Z)$ is $1$.  This is an example of a phenomenon known as systolic freedom and is discussed in \cite{CK1}. \\ 

Also, unlike in the proof of Theorem \ref{rpn_thm} for real projective space, where one knows that all length minimizing paths in a free homotopy class are smoothly immersed closed geodesics, the infima of the area and energy in a homotopy class of mappings $f:S^{2} \rightarrow (M,g)$, even if they are realized, may not be realized by a smooth immersion.  This is discussed by Sacks and Uhlenbeck in \cite{SU}, where they prove that if $(M,g)$ is a compact Riemannian manifold with $\pi_{2}(M) \neq 0$, there is a generating set for $\pi_{2}(M)$ which consists of conformal, branched minimal immersions of $S^{2}$ which minimize area and energy in their homotopy classes. \\  

When one has more information about the regularity of mappings which minimize area and energy in the homotopy class of $F_{*}(\C P^{1})$ in Theorem \ref{cpn_thm}, one may be able to draw stronger conclusions about the regularity of energy minimizing maps $F:(\C P^{N},g_{0}) \rightarrow (M,g)$.  This is the basis for our proof of Theorem \ref{cpn_holom_thm}.  

\begin{proof}[Proof of Theorem \ref{cpn_holom_thm}] Let $F:(\C P^{N},g_{0}) \rightarrow (X,h)$ be a Lipschitz mapping to a compact, simply connected K\"ahler manifold.  We will suppose $F_{*}(\C P^{1}) \in H_{2}(X;\Z)$ can be represented by a holomorphic curve -- the case in which $F_{*}(\C P^{1}) \in H_{2}(X;\Z)$ can be represented by an antiholomorphic curve follows by the same argument. \\ 

Suppose equality holds in Theorem \ref{cpn_thm} for $p = 2$.  This implies $F$ is smooth.  By the Hurewicz theorem, the natural mapping $\pi_{2}(X) \rightarrow H_{2}(X;\Z)$ is an isomorphism, so the family of mappings $f:S^{2} \rightarrow X$ which are homotopic to $F|_{\C P^{1}}$ coincides with the family of mappings which are homologous.  Letting $\omega_{h}$ be the K\"ahler form of the metric $h$, because $F_{*}([\C P^{1}]) \in H_{2}(X;\Z)$ can be represented by a holomorphic mapping $f:\C P^{1} \rightarrow X$, for $\mathcal{P} \in \mathcal{L}(\C P^{N})$ we have:   

\begin{equation}
\displaystyle \int\limits_{\mathcal{P}} F^{*}\omega_{h} = A^{\star} = \int\limits_{\mathcal{P}} |det(dF|_{\mathcal{P}})|. \bigskip 
\end{equation}

This implies that equality holds in the pointwise inequality $F^{*}\omega_{h{}} \leq |det(dF|_{\mathcal{P}})|$ at all $x \in \mathcal{P}$, which implies that $F:\mathcal{P} \rightarrow X$ is holomorphic. \\  

Choose affine coordinates $(z^{1}, z^{2}, \cdots, z^{N})$ on a neighborhood $\mathcal{U}$ of a point $x_{0}$ in $\C P^{N}$ and holomorphic coordinates on a neighborhood of $F(x_{0})$ in $X$.  For each $i = 1, 2, \dots, N$ and each fixed $z_{0}^{1}, z_{0}^{2}, \dots, z_{0}^{i-1}$, $z_{0}^{i+1}, \dots, z_{0}^{N}$, the affine line $(z_{0}^{1}, \dots, z^{i}, \dots, z_{0}^{N})$ in these coordinates on $\mathcal{U}$ corresponds to a complex subspace of complex dimension $2$ through the origin in $\C^{N+1}$, via the association $(z^{1}, z^{2}, \cdots, z^{N}) \rightarrow [1:z^{1}:z^{2}:\cdots :z^{N}]$ of affine and homogeneous coordinates on $\mathcal{U}$, and thus to a complex projective line $\mathcal{P}$ in $\C P^{N}$.  Because $F|_{\mathcal{P}}$ is holomorphic, each of the functions $F_{1}, F_{2}, \dots, F_{d}$ representing $F$ in these coordinates on $X$ is holomorphic when restricted to any such affine line in $\mathcal{U}$.  Osgood's Lemma \cite{GR_0} then implies that $F$ is holomorphic on $\mathcal{U}$ and therefore that $F$ is a holomorphic mapping of $\C P^{N}$. \\  

If $p > 2$ and equality holds for $p$, then by Theorem \ref{cpn_thm}, equality holds for all $p \in [2,\infty)$.  Equality for $p=2$ implies that $F$ is holomorphic.  Because $\int_{\C P^{1}}F^{*}\omega_{h} = A^{\star}$ for linear $\C P^{1} \subseteq \C P^{N}$, $F^{*}\omega_{h}$ is cohomologous to $(\frac{A^{\star}}{\pi})\widetilde{\omega}$, where $\widetilde{\omega}$ is the K\"ahler form of $g_{0}$.  This implies that, in the notation of Lemma \ref{elementary_lemma}, $Vol_{g_{0}}(\C P^{N},F^{*}h) = \frac{A^{\star^{N}}}{N!}$, and that for $p>2N$, the lower bound in Theorem \ref{cpn_thm} coincides with the lower bound in Lemma \ref{elementary_lemma}.  This implies that $F$ realizes equality in Lemma \ref{elementary_lemma} for $p > 2N$, and therefore that $F$ is a homothety onto its image. \end{proof}

Burns, Burstall, de Bartolomeis and Rawnsley \cite{BBdBR1} have proven that any stable harmonic map $\phi:(\C P^{N},g_{0}) \rightarrow (\mathcal{Z},h)$ to a compact, simple Hermitian symmetric space $(\mathcal{Z},h)$ is holomorphic or antiholomorphic, generalizing an earlier result of Ohnita \cite{Oh1} that holomorphic and antiholomorphic mappings are the only stable harmonic maps between complex projective spaces $(\C P^{N_{1}},g_{0})$ and $(\C P^{N_{2}},g_{0})$ with canonical metrics. 
 

\section{Quaternionic Projective Space}
\label{quaternions} 


In this section, we will prove Theorem \ref{hpn_thm}.  In the proof, we will make use of the twistor fibration $\Psi: \C P^{2N+1} \rightarrow \mathbb{H}P^{N}$.  This mapping, described in \cite{Sa1} and \cite[Ch. 14]{Be2}, gives a parametrization by $\C P^{2N+1}$ of the complex structures on tangent spaces to $\mathbb{H}P^{N}$ which satisfy a local compatibility condition with the the canonical metric $g_{0}$.  The results in Theorem \ref{hpn_thm} for $(\mathbb{H} P^{N},g_{0})$, $N \geq 2$ are different from the corresponding results for $(\mathbb{H} P^{1},g_{0})$, which is isometric to a rescaling of $(S^{4},g_{0})$.  These results for $(\mathbb{H} P^{1},g_{0})$ therefore follow from Lemma \ref{elementary_lemma}, and throughout the discussion below we will restrict to the case of $\mathbb{H} P^{N}$, $N \geq 2$ unless we explicitly state otherwise. \\ 

We fix a basis $1, I, J, K$ for the $\R$-algebra $\mathbb{H}$, satisfying the quaternion relations $I^{2} = J^{2} = K^{2} = IJK = -1$, and a Euclidean inner product on $\mathbb{H}^{N+1}$, which we view as a right $\mathbb{H}$-module.  Throughout, we will identify $\C$ with the subfield $\R + \R I$ of $\mathbb{H}$.  In this way we will view $\mathbb{H}^{N+1}$ and all of its $\mathbb{H}$-submodules and quotients as complex vector spaces.  Letting $U(1)$ denote the group of unit complex numbers and $Sp(1)$ the group of unit quaternions, we identify $\mathbb{H} P^{N}$ with the quotient $\mathbb{H}^{N+1} / \mathbb{H}^* = S^{4N + 3} / Sp(1)$, and we identify $\C P^{2N + 1}$ with $\mathbb{H}^{N+1} / \C^* = S^{4N + 3} / U(1)$.  Given a $1$-dimensional $\mathbb{H}$-submodule of $\mathbb{H}^{N+1}$, we will denote by $l$ both the associated point in $\mathbb{H} P^{N}$ and the subspace of $\mathbb{H}^{N+1}$.  We will write $Z_{l}$ for the fibre of the twistor fibration $\Psi : \C P^{2N + 1} \rightarrow \mathbb{H} P^{N}$ over $l$, which can be defined as follows: \\ 

For each point $p$ in $l \cap S^{4N+3}$, the differential of the Hopf fibration $\Phi:S^{4N+3} \rightarrow \mathbb{H}P^{N}$ gives an $\R$-linear isomorphism from $l^{\perp} \subseteq \mathbb{H}^{N+1}$ to $T_{l} \mathbb{H} P^{N}$.  This isomorphism can be used to transfer the $\C$-vector space and $\mathbb{H}$-module structures of $l^{\perp}$ to $T_{l} \mathbb{H} P^{N}$, but the structures induced on $T_{l} \mathbb{H} P^{N}$ in this way depend on the choice of the point $p$.  In particular, because the action of $Sp(1)$ on $l^{\perp}$ is not $U(1)$-equivariant, the complex structure on $T_{l} \mathbb{H} P^{N}$ induced in this way depends on $p$.  The subgroup of $Sp(1)$ which acts $U(1)$-equivariantly is $U(1)$ itself.  The complex structures induced on $T_{l} \mathbb{H} P^{N}$ in this way are therefore parametrized by $Sp(1)/U(1)$, which is canonically identified with the projectivization of $l \cong \C^{2}$ as a complex vector space.  The collection of complex structures induced on $T_{l}\mathbb{H}P^{N}$ as above, paramatrized by $Sp(1)/U(1)$, is the fibre $Z_{l}$ at $l$ of the twistor fibration $\Psi:\C P^{2N+1} \rightarrow \mathbb{H} P^{N}$. \\ 

Viewing $\C P^{2N+1}$ and $\mathbb{H}P^{N}$ as quotients of $S^{4N+3}$ by $U(1)$, resp. $Sp(1)$, the twistor fibration corresponds to the mapping $S^{4N+3}/U(1) \rightarrow S^{4N+3}/Sp(1)$.  The canonical metrics on $\C P^{2N+1}$ and $\mathbb{H} P^{N}$ are the base metrics of Riemannian submersions from $(S^{4N+3},g_{0})$ and the twistor fibration is itself a Riemannian submersion $(\C P^{2N+1},g_{0}) \rightarrow (\mathbb{H} P^{N},g_{0})$, with totally geodesic fibres $\C P^{1} \subseteq \C P^{2N+1}$.  Later it will be important that, although the $\C$-vector space and $\mathbb{H}$-module structures induced on $T_{l} \mathbb{H} P^{N}$ as above depend on the choice of $p \in l \cap S^{4N+3}$, some families of subspaces of $T_{l} \mathbb{H} P^{N}$ associated to these structures are independent of $p$.  We record these facts in the following: 

\begin{lemma}
\label{well_defined_lemma}
Let $l \in \mathbb{H} P^{N}$ and $\vec{v} \in T_{l}\mathbb{H} P^{N}$.  Then the $\mathbb{H}$-submodule of $T_{l}\mathbb{H} P^{N}$ generated by $\vec{v}$ in an $\mathbb{H}$-module structure induced on $T_{l}\mathbb{H} P^{N}$ as above, as a subspace of $T_{l}\mathbb{H} P^{N}$, is independent of $p$. \\ 

Likewise, for any complex structure $\tau$ on $T_{l} \mathbb{H} P^{N}$ which is induced by the twistor fibration as above and any complex $1$-dimensional subspace $\lambda$ of the complex vector space $(T_{l} \mathbb{H} P^{N},\tau)$, in any $\mathbb{H}$-module structure induced on $T_{l} \mathbb{H} P^{N}$ as above, $\lambda$ is contained in a $1$-dimensional $\mathbb{H}$-submodule of $T_{l} \mathbb{H} P^{N}$.  As a subspace of $T_{l} \mathbb{H} P^{N}$, this submodule is independent of the $\mathbb{H}$-module structure.  
\end{lemma}

We will denote the subspaces of $T_{l} \mathbb{H} P^{N}$ described in Lemma \ref{well_defined_lemma} by $\vec{v}\mathbb{H}$, resp. $\lambda\mathbb{H}$. 

\begin{proof} To show that $\vec{v}\mathbb{H}$ is well-defined, let $p_{1}$, $p_{2} \in l \cap S^{4N+3}$, and let $\vec{v}_{1},\vec{v}_{2} \in l^{\perp}$ such that $d\Phi_{p_i}(\vec{v}_{i}) = \vec{v}$, where $d\Phi$ is the differential of the Hopf fibration $\Phi : S^{4N+3} \rightarrow \mathbb{H} P^{N}$.  There is a unique $q \in Sp(1)$ such that $p_{2} = p_{1}q$, and $d\Phi_{p_{2}}^{-1} \circ d\Phi_{p_{1}}$ gives an $\R$-linear isomorphism of $l^{\perp}$, which is given by scalar multiplication by $q$.  We therefore have $\vec{v}_{2} = \vec{v}_{1}q$.  For $h \in \mathbb{H}$, $\vec{v}h$ in the $\mathbb{H}$-module structure induced via $p_{1}$ is: 

\begin{equation}
\label{well_defined_pf_eqn_1}
\displaystyle d\Phi_{p_{1}}(\vec{v}_{1}h) = d\Phi_{p_{2}} \circ d\Phi_{p_{2}}^{-1} \circ d\Phi_{p_{1}}(\vec{v}_{1}h) = d\Phi_{p_{2}}(\vec{v}_{1}hq) = d\Phi_{p_{2}}(\vec{v}_{2}q^{-1}hq), \bigskip 
\end{equation}
which corresponds to $\vec{v}q^{-1}hq$ in the $\mathbb{H}$-module structure induced via $p_{2}$.  This shows that the $\mathbb{H}$-submodule generated by $\vec{v}$ in the $\mathbb{H}$-module structure induced via $d\Phi_{p_{1}}$ is contained in the $\mathbb{H}$-submodule generated by $\vec{v}$ in the structure induced by $d\Phi_{p_{2}}$, and conversely.  \\ 

To see that the subspace $\lambda\mathbb{H}$ for a complex $1$-dimensional subspace $\lambda$ of $(T_{l} \mathbb{H} P^{N}, \tau)$ is well-defined for $\tau \in Z_{l}$, note that $\lambda \subseteq \vec{v}\mathbb{H}$ for any non-zero $\vec{v} \in \lambda$. \end{proof}

For $\tau \in Z_{l}$, we define the following family of bases of the complex vector space $(T_{l}\mathbb{H}P^{N},\tau)$: 

\begin{definition}
\label{goodframes}
Given a point $l$ in $\mathbb{H} P^{N}$ and $\tau \in Z_{l}$, let $\mathcal{F}_{l}(\tau)$ be the set of ordered bases $e_{1}, e_{2}, e_{3}, \cdots, e_{4N-1}, e_{4N}$ for $T_{l}\mathbb{H}P^{N}$, which are orthonormal in the canonical metric $g_{0}$, such that:  
\bigskip 
\begin{enumerate}
\item $e_{2j} = \tau(e_{2j-1})$ for all $j = 1, 2, \cdots, 2N$. 
\bigskip 

\item $e_{4i-3}, e_{4i-2}, e_{4i-1}, e_{4i}$ span a quaternionic line in $T_{l}\mathbb{H}P^{N}$ for each $i = 1, 2, \cdots, N$; that is, $e_{4i-3}\mathbb{H} = e_{4i-2}\mathbb{H} = e_{4i-1}\mathbb{H} = e_{4i}\mathbb{H}$. 
\bigskip 
\end{enumerate}
\end{definition}

We will write $\mathcal{F}(N)$ for the volume of the space of frames $\mathcal{F}_{l}(\tau)$ in Definition \ref{goodframes}, when viewed as a subset of the Steifel manifold of frames for $\R^{4N}$.  We then have:  

\begin{equation}
\displaystyle \mathcal{F}(1) = \sigma(3)\sigma(1) = 4\pi^{3}, \medskip 
\end{equation}
and for $N \geq 2$, 
\begin{equation}
\displaystyle \mathcal{F}(N) = Vol(\mathbb{H}P^{N-1})\mathcal{F}(N-1)\mathcal{F}(1). \bigskip 
\end{equation}

By convention, we define $\mathcal{F}(0) = 1$.  Because the volume of $(\mathbb{H}P^{N},g_{0})$ is equal to $\frac{\pi^{2N}}{(2N+1)!}$ for all $N \geq 1$, we then have: 

\begin{equation}
\label{Vol_F_formula}
\displaystyle \mathcal{F}(N) = \frac{\pi^{N(N+2)}4^{N}}{\prod_{j=1}^{N-1}(2j+1)!}. \bigskip 
\end{equation}

It will also be helpful to note that: 

\begin{equation}
\label{Vol_F_ratios}
\displaystyle  \frac{\mathcal{F}(1)\mathcal{F}(N-1)}{\mathcal{F}(N)} = \frac{(2N-1)!}{\pi^{2N-2}}, \medskip 
\end{equation}
\begin{equation}
\label{Vol_F_ratios_2}
\displaystyle  \frac{\mathcal{F}(N-2)}{\mathcal{F}(N)} = \frac{(2N-2)!(2N-3)!}{16 \pi^{4N}}. \bigskip 
\end{equation} 

Given a mapping $F:(\mathbb{H}P^{N},g_{0}) \rightarrow (M^{m},g)$, the following formula for the $4$-energy of $F$ at a point $l \in \mathbb{H}P^{N}$ is an elementary consequence of (\ref{energy_frame_eqn}):  

\begin{equation*}
\displaystyle e_{4}(F)_{l} = \frac{1}{\mathcal{F}(N)} \int\limits_{\mathcal{F}_{l}(\tau)} \left( \sum\limits_{i=1}^{2N} |dF(e_{2i-1})|^{2} + |dF(e_{2i})|^{2}\right)^{2} de \bigskip 
\end{equation*}
\begin{equation}
\label{H_formula}
\displaystyle = \frac{1}{\pi\mathcal{F}(N)} \int\limits_{Z_{l}}\int\limits_{\mathcal{F}_{l}(\tau)} \left( \sum\limits_{i=1}^{2N} |dF(e_{2i-1})|^{2} + |dF(e_{2i})|^{2}\right)^{2} de \ d\tau. \bigskip 
\end{equation}

For a complex structure $\tau$ induced on $T_{l} \mathbb{H} P^{N}$ via the twistor fibration, we define the following family of subspaces of $T_{l}\mathbb{H} P^{N}$: 
 
\begin{definition}
\label{hyperlagrangian}
Given $l \in \mathbb{H}P^{N}$ and $\tau \in Z_{l}$ as above, let $\mathcal{C}_{k}(\tau)$ be the family of complex subspaces $V$ of the complex vector space $(T_{l} \mathbb{H} P^{N},\tau)$ which are of complex dimension $k$ and which have the following property:  for any complex $1$-dimensional subspace $\lambda$ of $V$, the orthogonal complement of $\lambda$ in the quaternionic line $\lambda\mathbb{H}$ is also orthogonal to $V$.  
\end{definition}

We will adapt the notation introduced for $\mathbb{H}$-modules generated by elements and subsets of $T_{l}\mathbb{H} P^{N}$ above:  for a complex subspace $\widetilde{\lambda}$ of $\mathbb{H}^{N+1}$, we will write $\widetilde{\lambda}\mathbb{H}$ for the $\mathbb{H}$-module generated by $\widetilde{\lambda}$. 

\begin{lemma}
\label{hyperlagrangian_lemma}
Let $\tau \in Z_{l}$ and $V \in C_{k}(\tau)$.  Then there is a unique complex subspace $\widetilde{V} \subseteq l^{\perp} \subseteq \mathbb{H}^{N+1}$ such that, for any $p \in l \cap S^{4N+3}$ which induces the complex structure $\tau$ via the differential $d\Phi_{p}:l^{\perp} \rightarrow T_{l}\mathbb{H}P^{N}$ of the Hopf fibration as above, $d\Phi_{p}^{-1}(V) \cap l^{\perp} = \widetilde{V}$.  For any complex subspace $\widetilde{\lambda}$ of $\widetilde{V}$ of complex dimension $1$, the orthogonal complement of $\widetilde{\lambda}$ in the quaternionic line $\widetilde{\lambda} \mathbb{H}$ which it generates in $\mathbb{H}^{N+1}$ is orthogonal to $\widetilde{V}$.  
\end{lemma}

\begin{proof} For a given $p_{0} \in l \cap S^{4N+3}$ which induces the almost-complex structure $\tau$, let $\widetilde{V}_{p_{0}}$ be be the horizontal pre-image of $V$ via the Hopf fibration $\Phi:S^{4N+3} \rightarrow \mathbb{H} P^{N}$.  If $p_{1}$ is another point in $l \cap S^{4N+3}$ which induces $\tau$, then the unique $q \in Sp(1)$ with $p_{1} = p_{0}q$ in fact belongs to $U(1)$.  The complex subspace $\widetilde{V}_{p_{0}}$ is preserved by multiplication by $q$ and therefore coincides with the horizontal lift of $V$ at $p_{1}$.  This common horizontal preimage is $\widetilde{V}$.  The orthogonality of $\widetilde{\lambda}^{\perp} \subseteq \widetilde{\lambda}\mathbb{H}$ and $\widetilde{V}$ is equivalent to the same property for complex $1$-dimesional subspaces $\lambda$ of $V$. \end{proof} 

For each $V \in \mathcal{C}_{2}(\tau)$, the $\mathbb{H}$-module generated by $V$ in an $\mathbb{H}$-module structure induced on $T_{l}\mathbb{H} P^{N}$ as above is again a well-defined subspace, of real dimension $8$, which is an $\mathbb{H}$-submodule for any $\mathbb{H}$-module structure induced on $T_{l}\mathbb{H}P^{N}$ as above, as in Lemma \ref{well_defined_lemma}.  We will denote this $V\mathbb{H}$.  $\mathcal{C}_{2}(\tau)$ is the base of a fibration $\mathcal{F}_{l}(\tau) \rightarrow \mathcal{C}_{2}(\tau)$, which sends a frame $e_{1}, e_{2}, e_{3}, \cdots, e_{4N-1}, e_{4N}$ to the $\R$-span of $e_{1}, e_{2}, e_{5}, e_{6}$.  We will equip $\mathcal{C}_{2}(\tau)$ with the measure pushed forward from $\mathcal{F}_{l}(\tau)$ via this mapping.  We then have:  

\begin{equation}
\label{Vol_C_eqn}
\displaystyle Vol(\mathcal{C}_{2}(\tau)) = \frac{\mathcal{F}(N)}{\sigma(3)\sigma(1)^{3}\mathcal{F}(N-2)} = \frac{\pi^{4N-5}}{(2N-2)!(2N-3)!}, \bigskip 
\end{equation}
where we have used (\ref{Vol_F_ratios_2}), the fact that the space of ordered frames $e_{1}, e_{2}, e_{5}, e_{6}$ for $V$ has volume $\sigma(3)\sigma(1)$ and the fact that for each such ordered frame, the volume of the space of frames $e_{3}, e_{4}, e_{7}, e_{8}$ for $V^{\perp} \subseteq V\mathbb{H}$ which together with $e_{1}, e_{2}, e_{5}, e_{6}$ give a frame $e_{1}, e_{2}, \cdots, e_{7}, e_{8}$ for $V\mathbb{H}$ as in Definition \ref{goodframes} is equal to $\sigma(1)^{2}$.  We will let $Gr_{1}^{\mathbb{H}}(l)$ denote the family of quaternionic $1$-dimensional subspaces of $T_{l}\mathbb{H}P^{N}$, that is, the family of $1$-dimensional $\mathbb{H}$-modules determined by $\vec{v} \in T_{l}\mathbb{H} P^{N}$ as in Lemma \ref{well_defined_lemma} -- this space is a copy of $\mathbb{H}P^{N-1}$, and we equip it with the canonical measure pushed forward from the measure on the unit sphere in $T_{l}\mathbb{H} P^{N}$ via the Hopf fibration.  The fundamental pointwise result which is the basis for Theorem \ref{hpn_thm} is: 

\begin{lemma}
\label{twistor_lemma}
Let $F:(\mathbb{H}P^{N},g_{0}) \rightarrow (M^{m},g)$ be a Lipschitz mapping and $l \in \mathbb{H}P^{N}$ a point at which $F$ is differentiable.  Then: 

\begin{equation}
\label{twistor_lemma_eqn}
\displaystyle e_{4}(F)_{l} \geq \scriptstyle \frac{16N^{2}(2N-2)!}{\pi^{2N-2}} \displaystyle \left( \int\limits_{Gr_{1}^{\mathbb{H}}(l)} |\det(dF|_{\Lambda})| d\Lambda + \scriptstyle \frac{(2N-2)!}{(2N-1)\pi^{2N-2}} \displaystyle \iint\limits_{Z_{l} \ \mathcal{C}_{2}(\tau)} |\det(dF|_{V})| dV d\tau \right). 
\end{equation}

Equality holds if and only $dF_{l}$ is a homothety. 
\end{lemma}

\begin{proof} We begin with the identity for $e_{4}(F)_{l}$ from (\ref{H_formula}).  By Newton's inequality for symmetric functions (cf. \cite[p.113]{ES1}) and the arithmetic-geometric mean inequality, 

\begin{equation*}
\displaystyle e_{4}(F)_{l} \geq \scriptstyle \frac{16N}{\pi(2N-1)\mathcal{F}(N)} \displaystyle \int\limits_{Z_{l}}\int\limits_{\mathcal{F}_{l}(\tau)} \sum\limits_{i=1}^{2N-1} \sum\limits_{j=i+1}^{2N}  |dF(e_{2i-1})||dF(e_{2i})|dF(e_{2j-1})||dF(e_{2j})| de d\tau 
\end{equation*}
\begin{equation*}
\displaystyle = \scriptstyle \frac{16N}{\pi(2N-1)\mathcal{F}(N)} \displaystyle \text{\Huge{[}} \int\limits_{Z_{l}}\int\limits_{\mathcal{F}_{l}(\tau)} \sum\limits_{i=1}^{N} |dF(e_{4i-3})||dF(e_{4i-2})||dF(e_{4i-1})||dF(e_{4i})| de d\tau 
\end{equation*}
\begin{equation*}
\displaystyle + \int\limits_{Z_{l}}\int\limits_{\mathcal{F}_{l}(\tau)} \text{\Large{(}} \sum\limits_{j=1}^{N-1} \sum\limits_{k=j+1}^{N} |dF(e_{4j-3})||dF(e_{4j-2})||dF(e_{4k-3})||dF(e_{4k-2})| 
\end{equation*}
\begin{equation*}
\displaystyle + \sum\limits_{j=1}^{N-1} \sum\limits_{k=j+1}^{N} |dF(e_{4j-3})||dF(e_{4j-2})||dF(e_{4k-1})||dF(e_{4k})| 
\end{equation*}
\begin{equation*}
\displaystyle + \sum\limits_{j=1}^{N-1} \sum\limits_{k=j+1}^{N} |dF(e_{4j-1})||dF(e_{4j})||dF(e_{4k-3})||dF(e_{4k-2})| 
\end{equation*}
\begin{equation}
\label{twistor_lemma_pf_eqn_1}
\displaystyle + \sum\limits_{j=1}^{N-1} \sum\limits_{k=j+1}^{N} |dF(e_{4j-1})||dF(e_{4j})||dF(e_{4k-1})||dF(e_{4k})| \text{\Large{)}} de d\tau \text{\Huge{]}}. \bigskip  
\end{equation}

For each $\Lambda \in Gr_{1}^{\mathbb{H}}(l)$ and each orthonormal frame $e_{4i-3}, e_{4i-2}, e_{4i-1}, e_{4i}$ for $\Lambda$, we have:  

\begin{equation}
\label{twistor_lemma_pf_eqn_2}
\displaystyle |dF(e_{4i-3})||dF(e_{4i-2})||dF(e_{4i-1})||dF(e_{4i})| \geq |\det(dF|_{\Lambda})|. \bigskip 
\end{equation}	

Therefore, for each $\tau \in Z_{l}$, 

\begin{equation*}
\displaystyle \int\limits_{\mathcal{F}_{l}(\tau)} \sum\limits_{i=1}^{N} |dF(e_{4i-3})||dF(e_{4i-2})||dF(e_{4i-1})||dF(e_{4i})| de 
\end{equation*} 
\begin{equation}
\label{twistor_lemma_pf_eqn_3}
\displaystyle \geq N \mathcal{F}(1)\mathcal{F}(N-1) \int\limits_{Gr_{1}^{\mathbb{H}}(l)} |\det(dF|_{\Lambda})| d\Lambda, \bigskip  	
\end{equation}

Note that the constant $N \mathcal{F}(1)\mathcal{F}(N)$ in (\ref{twistor_lemma_pf_eqn_3}) is the product of the number of indices $i = 1, 2, \cdots, N$ in the summation, the volume $\mathcal{F}(1)$ of the space of frames for $\Lambda$ and the volume $\mathcal{F}(N-1)$ of the space of frames for $\Lambda^{\perp}$. \\ 

Similarly, for each $\tau \in Z_{l}$ and each $V \in \mathcal{C}_{2}(\tau)$, the space of frames $e_{4j-3}, e_{4j-2}, e_{4k-3}, e_{4k-2}$ for $V$ has volume $\sigma(3)\sigma(1) = 4\pi^{3}$.  The volume of the space of frames $e_{4j-1}, e_{4j}, e_{4k-1}, e_{4k}$ for $V^{\perp} \subseteq V\mathbb{H}$ has volume $\sigma(1)^{2} = 4\pi^{2}$, the space of frames for $V\mathbb{H}^{\perp}$ has volume $\mathcal{F}(N-2)$ and there are $\binom{N}{2}$ terms in the summation $\sum_{j=1}^{N-1} \sum_{k=j+1}^{N}$ at which the frame may occur.  The same is true for the other possible indexings of each frame, i.e. $e_{4j-3}, e_{4j-2}, e_{4k-1}, e_{4k}$; $e_{4j-1}, e_{4j}, e_{4k-3}, e_{4k-2}$ and $e_{4j-1}, e_{4j}, e_{4k-1}, e_{4k}$.  For each $\tau \in Z_{l}$, the second integral term in (\ref{twistor_lemma_pf_eqn_1}) is therefore bounded below by: 

\begin{equation}
\label{twistor_lemma_pf_eqn_4}
\displaystyle 32N(N-1)\pi^{5}\mathcal{F}(N-2) \int\limits_{\mathcal{C}_{2}(\tau)} |\det(dF|_{V})| dV. \bigskip 
\end{equation}

Combining (\ref{twistor_lemma_pf_eqn_1}) with the lower bounds in (\ref{twistor_lemma_pf_eqn_3}) and (\ref{twistor_lemma_pf_eqn_4}), noting that the integration over $Z_{l}$ in (\ref{twistor_lemma_pf_eqn_1}) adds a factor $\pi$ to the constant in (\ref{twistor_lemma_pf_eqn_3}) and using the identities (\ref{Vol_F_ratios}) gives (\ref{twistor_lemma_eqn}).  Equality requires equality in Newton's inequality and the arithmetic-geometric mean inequality in (\ref{twistor_lemma_pf_eqn_1}).  This implies that for all $\tau \in Z_{l}$ and all $\lbrace e_{1}, e_{2}, \dots, e_{4N} \rbrace \in \mathcal{F}_{l}(\tau)$, for all $i,j = 1, 2, \cdots 2N$, $|dF(e_{2i-1})|^{2} + |dF(e_{2i})|^{2} = |dF(e_{2j-1})|^{2} + |dF(e_{2j})|^{2}$ and $|dF(e_{2i-1})| = |dF(e_{2i})|$.  This implies $|dF(\vec{u})|$ is the same for all unit vectors $\vec{u}$ tangent to $\mathbb{H}P^{N}$ at $l$, so that $dF_{l}$ is a homothety.  \end{proof}

For each $\tau \in Z_{l}$ and each $V \in \mathcal{C}_{2}(\tau)$, $V$ is tangent to a totally geodesic submanifold of $\mathbb{H}P^{N}$ which is isometric to $\C P^{2}$ with its canonical metric, cf. \cite[Ch. 5]{Be1}.  In the following lemma, we will describe a construction of these totally geodesic submanifolds and establish some of their properties which we will use in the proof of Theorem \ref{hpn_thm}:  

\begin{lemma}
\label{tot_geo_lemma}

Let $l \in \mathbb{H} P^{N}$.  Let $\tau \in Z_{l}$ and $V \in C_{2}(\tau)$, and let $\widetilde{V} \subseteq l^{\perp}$ be the subspace associated to $V$ as in Lemma \ref{hyperlagrangian_lemma}.  Let $\lambda_{\tau}$ be the complex $1$-dimensional subspace of $l \cong \C^{2}$ which corresponds to the complex structure $\tau$, as a point in the projectivization of $l$ as a complex vector space, via the twistor fibration.  

\begin{enumerate}
	\item 
	\label{tot_geo_1}
	For any complex subspace $\lambda'$ of $\widetilde{V} \oplus \lambda_{\tau} \subseteq \mathbb{H}P^{N+1}$ of complex dimension $1$, the orthogonal complement to $\lambda'$ in $\lambda'\mathbb{H}$ is orthogonal to $\widetilde{V} \oplus \lambda_{\tau}$ in $\mathbb{H}^{N+1}$.  
	\medskip 
	
	\item 
	\label{tot_geo_2}
	Let $\widetilde{X}_{\tau}$ be the image of $\widetilde{V} \oplus \lambda_{\tau}$ in $\C P^{2N+1}$ via the projectivization of $\mathbb{H}^{N+1} = \C^{2N+2}$ as a complex vector space.  Then $\widetilde{X}_{\tau}$ maps injectively to $\mathbb{H} P^{N}$ via the twistor fibration, to a totally geodesic submanifold which we will denote $X_{\tau}$, and $V$ is the tangent space to $X_{\tau}$ at $l$. 
	\medskip 
		
	\item 
	\label{tot_geo_4}
	For each complex $1$-dimensional subspace $\lambda'$ of $\widetilde{V} \oplus \lambda_{\tau}$, let $l' = \lambda' \mathbb{H} \subseteq \mathbb{H}^{N+1}$ and let $\tau'$ be the complex structure on $T_{l'}\mathbb{H}P^{N}$ associated to $\lambda'$ via the twistor fibration, when we view $\lambda'$ as a point in the projectivization of $l'$ as a complex vector space.  Then $T_{l'}X_{\tau}$ is $\tau'$-invariant.  In particular, the submanifold $\widetilde{X}_{\tau}$ of $\C P^{2N+1}$ is a section of the twistor bundle over $X_{\tau}$ and the induced complex structure on $T\mathbb{H}P^{N}|_{X_{\tau}}$ preserves the tangent bundle $TX_{\tau}$ of $X_{\tau}$ and the normal bundle to $X_{\tau}$ in $\mathbb{H} P^{N}$.
\end{enumerate}
\end{lemma}

\begin{proof} {\em (\ref{tot_geo_1}):} For any $\cos (\theta) J + \sin (\theta) K \in Sp(1)$, and for any $v \in \widetilde{V}$, we have $v \left( \cos (\theta) J + \sin (\theta) K \right)$ orthogonal to both $\widetilde{V}$ and $\lambda_{\tau}$.  Likewise, for any $w \in \lambda_{\tau}$, $w \left( \cos (\theta) J + \sin (\theta) K \right)$ is orthogonal to both $\lambda_{\tau}$ and $\widetilde{V}$.  This implies that the same is true for any vector $u$ in $\widetilde{V} \oplus \lambda_{\tau}$, and therefore, for the complex $1$-dimensional subspace $\lambda' \subseteq \mathbb{H}^{N+1}$ spanned by $u$. \\

{\em (\ref{tot_geo_2}):} By Part (\ref{tot_geo_1}), if $v,w \in \widetilde{V} \oplus \lambda_{\tau}$ belong to the same $Sp(1)$ orbit in $\mathbb{H}^{N+1} \setminus \lbrace 0 \rbrace$, they in fact belong to the same $U(1)$ orbit.  This implies that the image of $\widetilde{V} \oplus \lambda_{\tau}$ in $\C P^{2N+1}$ maps injectively to $\mathbb{H} P^{N}$ via the twistor fibration $S^{4N+3}/U(1) \rightarrow S^{4N+3}/Sp(1)$.  The submanifold $\widetilde{X}_{\tau}$ of $\C P^{2N+1}$ is horizontal for the twistor fibration; that is, $\widetilde{X}_{\tau}$ meets the fibres of the twistor mapping orthogonally -- this implies that $\widetilde{X}_{\tau}$ maps isometrically to its image via the twistor fibration.  To see that $\widetilde{X}_{\tau}$ is horizontal, note that for any complex $1$-dimensional subspace $\lambda'$ of $\widetilde{V} \oplus \lambda_{\tau}$, by Part (\ref{tot_geo_1}), the orthogonal complement to $\lambda'$ in $\widetilde{V} \oplus \lambda_{\tau}$ is orthogonal to $l'$.  This orthogonal complement $\lambda'^{\perp} \subseteq \widetilde{V} \oplus \lambda_{\tau}$ gives the horizontal lift of $T_{\lambda'}\widetilde{X}_{\tau}$ for the Hopf fibration $S^{4N+3} \rightarrow \C P^{2N+1}$ at any point of $\lambda' \cap S^{4N+3}$, and the horizontal lift of the tangent space to the twistor fibre at such a point is contained in $l'$.  This also implies that $V = T_{l} X_{\tau}$. \\  

$\widetilde{X}_{\tau}$ is a linearly embedded $\C P^{2}$ in $\C P^{2N+1}$ and thus is totally geodesic.  To see that $X_{\tau}$ is totally geodesic in $\mathbb{H} P^{N}$, note that for any vector fields $V,W$ on $\mathbb{H} P^{N}$ which are tangent to $X_{\tau}$ along $X_{\tau}$, their horizontal lifts $\widetilde{V},\widetilde{W}$ are tangent to $\widetilde{X}_{\tau}$ along $\widetilde{X}_{\tau}$, and $[\widetilde{V},\widetilde{W}]$ is therefore tangent to $\widetilde{X}_{\tau}$ along $\widetilde{X}_{\tau}$.  In particular, this implies that $[\widetilde{V},\widetilde{W}]$ is horizontal for the twistor fibration along $\widetilde{X}_{\tau}$.  Ths implies via O'Neill's formula that the second fundamental form of $X_{\tau}$ in $\mathbb{H} P^{N}$ is the same as that of $\widetilde{X}_{\tau}$ in $\C P^{2N+1}$, under the natural identification of the normal bundle of $X_{\tau}$ in $\mathbb{H} P^{N}$ with its horizontal pre-image, and therefore that $X_{\tau}$ is totally geodesic in $\mathbb{H} P^{N}$. \\ 

{\em (\ref{tot_geo_4}):} Because the orthogonal complement $\lambda'^{\perp}$ of $\lambda'$ in $\widetilde{X} \oplus \lambda_{\tau}$ is a complex subspace of $l'^{\perp} \subseteq \mathbb{H}^{N+1}$, its image via the differential of the Hopf fibration $S^{4N+3} \rightarrow \mathbb{H} P^{N}$ is a complex subspace of $T_{l'}\mathbb{H} P^{N}$ in the complex structure associated to $\lambda'$ as a point in the projectivization of $l' \cong \C^{2}$ via the twistor fibration.  We have seen in Part (\ref{tot_geo_2}) that this image is the tangent space to $X_{\tau}$ at $l'$. \end{proof} 

We will write $\mathcal{C}(\mathbb{H}P^{N})$ for the space of all such totally geodesic $X_{\tau}$ in $\mathbb{H}P^{N}$.  We will equip $\mathcal{C}(\mathbb{H}P^{N})$ with the measure pushed forward from the total space of the bundle over $\C P^{2N+1}$ whose fibre over $\tau$ is $\mathcal{C}_{2}(\tau)$, via the natural fibration $V \mapsto X_{\tau}$ of this space over $\mathcal{C}(\mathbb{H}P^{N})$.  We then have: 
 
\begin{equation*}
\displaystyle Vol \left(\mathcal{C}(\mathbb{H}P^{N})\right) = \frac{Vol(\C P^{2N+1},g_{0})Vol(\mathcal{C}_{2}(\tau))}{Vol(\C P^{2})} \bigskip 
\end{equation*}
\begin{equation}
\label{Vol_tot_geo_space}
\displaystyle = \frac{2\pi^{6N-6}}{(2N+1)!(2N-2)!(2N-3)!}. \bigskip 
\end{equation}

We will let $\mathcal{H}(\mathbb{H}P^{N})$ be the space of linearly embedded $\mathbb{H}P^{1} \subseteq \mathbb{H}P^{N}$, with the measure pushed forward from the unit tangent bundle $U(\mathbb{H}P^{N},g_{0})$.  Then: 

\begin{equation}
\displaystyle Vol\left(\mathcal{H}(\mathbb{H}P^{N})\right) = \frac{Vol\left(U(\mathbb{H}P^{N},g_{0})\right)}{Vol\left(U(\mathbb{H}P^{1},g_{0})\right)} = \frac{6\pi^{4N-4}}{(2N+1)!(2N-1)!}. \bigskip 
\end{equation}

The constant $K_{N}$ in Theorem \ref{hpn_thm} can be defined in terms of $\mathcal{H}(\mathbb{H}P^{N})$ and $\mathcal{C}(\mathbb{H}P^{N})$:  

\begin{equation*}
\displaystyle K_{N} = \frac{16N^{2}(2N+1)!(2N-2)!}{\pi^{4N-2}} \left( Vol\left(\mathcal{H}(\mathbb{H}P^{N})\right) + \frac{(2N-2)!}{(2N-1)\pi^{2N-2}} Vol \left(\mathcal{C}(\mathbb{H}P^{N})\right) \right) \bigskip 
\end{equation*}
\begin{equation}
\label{hpn_const_formula}
\displaystyle = \frac{32N^{2}(2N+1)}{\pi^{2}(2N-1)}. \bigskip 
\end{equation}

Given a triple of complex structures $I,J,K$ induced by locally defined sections of the twistor bundle on a neighborhood of $\mathbb{H}P^{N}$, which satisfy the quaternion relations $I^{2} = J^{2} = K^{2} = IJK = -Id$, we can form their associated K\"ahler forms $\omega_{I}$, $\omega_{J}$, $\omega_{K}$ with the canonical metric $g_{0}$ on $\mathbb{H} P^{N}$.  The 4-form $\omega_{I}^{2} + \omega_{J}^{2} + \omega_{K}^{2}$ is independent of the choice of $I,J,K$.  This form therefore coincides with a canonical, globally-defined 4-form on $\mathbb{H}P^{N}$, known as the fundamental 4-form or Kraines 4-form, whose powers generate the cohomology of $\mathbb{H}P^{N}$ and calibrate linear subspaces $\mathbb{H}P^{k} \subseteq \mathbb{H}P^{N}$, cf. \cite{Be1,Kr1,Kr2}.  More precisely, we define $\Omega$ to be the form which coincides with $\frac{1}{\pi^{2}}\left( \omega_{I}^{2} + \omega_{J}^{2} + \omega_{K}^{2} \right)$ for any choice of $I,J,K$ as above.  $\Omega$ is a closed, parallel form satisfying $\langle \Omega, \mathbb{H}P^{1} \rangle = 1$.  In $H^{4}(\mathbb{H}P^{N};\R)$, the cohomology class of $\Omega$ is the image of a generator of $H^{4}(\mathbb{H}P^{N};\Z)$ via the natural homomorphism $H^{4}(\mathbb{H}P^{N};\Z) \rightarrow H^{4}(\mathbb{H}P^{N};\R)$.  For any orthonormal frame $e,I(e),J(e),K(e)$ for a quaternionic line in $T_{l}\mathbb{H}P^{N}$, $\Omega(e,I(e),J(e),K(e)) = \frac{6}{\pi^{2}}$.  More generally, $(\frac{\pi^{2}}{6})\Omega$ has comass $\equiv 1$ and gives a calibration of $(\mathbb{H}P^{N},g_{0})$, whose calibrated submanifolds are precisely the linearly embedded $\mathbb{H}P^{1}$ in $\mathbb{H}P^{N}$.  The powers of $\Omega$, appropriately rescaled, likewise give calibrations whose calibrated submanifolds are linear subspaces $\mathbb{H}P^{d} \subseteq \mathbb{H} P^{N}$.  If $X \in \mathcal{C}(\mathbb{H}P^{N},g_{0})$, then by choosing $I$ to coincide with the complex structure $\tau$ on $T_{l}X$ along $X$ as in Lemma \ref{tot_geo_lemma} (5), we have that for any orthonormal frame $e_{1}, I(e_{1}), e_{2}, I(e_{2})$ for $T_{l}X$, $\Omega(e_{1}, I(e_{1}), e_{2}, I(e_{2})) = \frac{2}{\pi^{2}}$.  Since $Vol(\C P^{2}) = \frac{\pi^{2}}{2} = 3Vol(\mathbb{H}P^{1})$, this implies that $X$ also represents a generator of $H_{4}(\mathbb{H}P^{N};\Z)$ and is ``one third calibrated" by $\Omega$.  

\begin{proof}[Proof of Theorem \ref{hpn_thm}] Let $F:(\mathbb{H}P^{N},g_{0}) \rightarrow (M^{m},g)$ be a mapping as above and $p \geq 4$.  By Lemma \ref{twistor_lemma}, 

\begin{equation*}
\displaystyle E_{p}(F) = \int\limits_{\mathbb{H}P^{N}} |dF|^{p} dVol_{g_{0}} \bigskip        
\end{equation*}
\begin{equation}
\label{hpn_thm_pf_eqn_1}
\displaystyle \geq \scriptstyle \left( \frac{16N^{2}(2N-2)!}{\pi^{2N-2}} \right)^{\frac{p}{4}} \displaystyle \int\limits_{\mathbb{H}P^{N}} \left( \int\limits_{Gr_{1}^{\mathbb{H}}(l)} |\det(dF|_{\Lambda})| d\Lambda + \scriptstyle  \frac{(2N-2)!}{(2N-1)\pi^{2N-2}} \displaystyle \iint\limits_{Z_{l} \ \mathcal{C}_{2}(\tau)} |\det(dF|_{V})| dV \right)^{\frac{p}{4}} dVol_{g_{0}}. \bigskip 
\end{equation}

For $p = 4$ this immediately implies that $E_{p}(F)$ is bounded below by: 

\begin{equation}
\label{hpn_thm_pf_eqn_2}
\scriptstyle \left( \frac{16N^{2}(2N-2)!}{\pi^{2N-2}} \right) \displaystyle \left( \int\limits_{\mathcal{H}(\mathbb{H}P^{N})} |F(\mathcal{Q})| d\mathcal{Q} + \scriptstyle \frac{(2N-2)!}{(2N-1)\pi^{2N-2}} \displaystyle \int\limits_{\mathcal{C}(\mathbb{H}P^{N})} |F(X)| dX \right). \bigskip 
\end{equation}

For $p > 4$, (\ref{hpn_thm_pf_eqn_1}) and H\"older's inequality imply that $E_{p}(F)$ is bounded below by: 

\begin{equation*}
\label{hpn_thm_pf_eqn_3}
\scriptstyle \left(  \frac{((2N+1)!)^{p-4}(16N^{2}(2N-2)!)^{p}}{\pi^{(4N-2)p - 8N}}   \right)^{\frac{1}{4}}  \displaystyle \left( \int\limits_{\mathcal{H}(\mathbb{H}P^{N})} |F(\mathcal{Q})| d\mathcal{Q} + \scriptstyle \frac{(2N-2)!}{(2N-1)\pi^{2N-2}} \displaystyle \int\limits_{\mathcal{C}(\mathbb{H}P^{N})}|F(X)| dX \right)^{\frac{p}{4}}. \bigskip 
\end{equation*}

Because all $\mathcal{Q} \in \mathcal{H}(\mathbb{H}P^{N},g_{0})$ and $X \in \mathcal{C}(\mathbb{H}P^{N},g_{0})$ represent generators of $H_{4}(\mathbb{H} P^{N};\Z)$ we have $|F(\mathcal{Q})|, |F(X)| \geq B^{\star}$, which implies the inequality (\ref{hpn_thm_eqn}), albeit with nonstrict rather than strict inequality. \\ 

For (\ref{hpn_thm_eqn}) to be an equality, equality would have to hold a.e. in (\ref{hpn_thm_pf_eqn_1}), and therefore in Lemma \ref{twistor_lemma}.  This would imply that $F^{*}g$ is equal a.e. to $\varphi(x)g_{0}$, where $\varphi(x)$ is an a.e.-defined non-negative function on $\mathbb{H}P^{N}$.  Equality would also require that $F$ take almost all $\mathcal{Q},X \subseteq \mathbb{H}P^{N}$ to area minimizing currents in their homology class in $H_{4}(M;\Z)$.  Together, these conditions would imply that for almost all $\mathcal{Q} \in \mathcal{H}(\mathbb{H}P^{N})$ and almost all $X \in \mathcal{C}(\mathbb{H}P^{N})$, $\varphi$ is defined a.e. on $\mathcal{Q}$, resp. $\mathcal{X}$ and  

\begin{equation}
\label{hpn_thm_pf_eqn_4}
\displaystyle \int\limits_{\mathcal{Q}} \varphi(x)^{2} dx = \int\limits_{X} \varphi(x)^{2} dx = B^{\star}, \bigskip 
\end{equation}
where the integration in (\ref{hpn_thm_pf_eqn_4}) is with respect to the volume form of $g_{0}|_{\mathcal{Q}}$, resp. $g_{0}|_{X}$.  Letting $\zeta: U(\mathbb{H}P^{N},g_{0}) \rightarrow \mathbb{H}P^{N}$ be the bundle projection of the unit tangent bundle of $\mathbb{H}P^{N}$, we would then have:  

\begin{equation*}
\displaystyle \int\limits_{\mathbb{H}P^{N}} \varphi(l)^{2} dVol_{g_{0}} = \frac{1}{\sigma(4N-1)} \int\limits_{U(\mathbb{H}P^{N},g_{0})} \varphi(\zeta(\vec{u}))^{2} d\vec{u} \bigskip 
\end{equation*}
\begin{equation*}
\displaystyle = \frac{1}{\sigma(4N-1)} \int\limits_{\mathcal{H}(\mathbb{H}P^{N})} \int\limits_{U(\mathcal{Q},g_{0})} \varphi(\zeta(\vec{u}))^{2} d\vec{u} d\mathcal{Q} \bigskip 
\end{equation*}
\begin{equation}
\label{hpn_thm_pf_eqn_5}
\displaystyle = \frac{\sigma(3)}{\sigma(4N-1)} \int\limits_{\mathcal{H}(\mathbb{H}P^{N})} \int\limits_{\mathcal{Q}} \varphi(x)^{2} dx d\mathcal{Q} = \left(\frac{6\pi^{2N-2}}{(2N+1)!}\right) B^{\star}. \bigskip 
\end{equation}

On the other hand, we would also have: 

\begin{equation*}
\displaystyle \int\limits_{\mathbb{H}P^{N}} \varphi(l)^{2} dVol_{g_{0}} = \frac{1}{\pi} \int\limits_{\C P^{2N+1}} \varphi(\Psi(\tau))^{2} d\tau \bigskip   
\end{equation*}
\begin{equation*}
\displaystyle = \frac{(2N-2)!(2N-3)!}{\pi^{4N-4}} \int\limits_{\C P^{2N+1}} \int\limits_{\mathcal{C}_{2}(\tau)} \varphi(\Psi(\tau))^{2} dV d\tau \bigskip 
\end{equation*}
\begin{equation}
\label{hpn_thm_pf_eqn_6}
\displaystyle = \frac{(2N-2)!(2N-3)!}{\pi^{4N-4}} \int\limits_{\mathcal{C}(\mathbb{H}P^{N})} \int\limits_{X} \varphi(x)^{2} dxdX = \left( \frac{2\pi^{2N-2}}{(2N+1)!} \right) B^{\star}. \bigskip 
\end{equation}

The right-hand sides of (\ref{hpn_thm_pf_eqn_5}) and (\ref{hpn_thm_pf_eqn_6}) can only be equal if $B^{\star} = 0$, which implies that equality cannot hold in the inequality in Theorem \ref{hpn_thm} \end{proof}


A theorem of White \cite{Wh2} shows that if $F:(W^{d},g) \rightarrow (N^{n},h)$ is a mapping of compact, connected, oriented Riemannian manifolds with $F_{*}: \pi_{1}(W) \rightarrow \pi_{1}(N)$ surjective and $\dim(W) \geq 3$, then the infimum of the areas of mappings homotopic to $F$ is equal to the minimum mass of an integral current $T$ representing $F([W])$ in $H_{d}(N;\Z)$.  If $M$ in Theorem \ref{hpn_thm} is simply connected, $B^{\star}$ is therefore equal to the infimum of the areas of mappings $f:S^{4} \rightarrow M$ in the free homotopy class of $F_{*}(\mathbb{H}P^{1})$, as in Theorems \ref{rpn_thm} for $\R P^{n}$ and \ref{cpn_thm} for $\C P^{N}$. \\ 

For $p \geq 4N$ Lemma \ref{elementary_lemma} implies that the identity mapping of $(\mathbb{H} P^{N},g_{0})$ is $p$-energy minimizing in its homotopy class, so the result given by Theorem \ref{hpn_thm} is not optimal in this setting.  However for all $p \geq 4$, the proof of Theorem \ref{hpn_thm} implies that the identity mapping of $(\mathbb{H}P^{N},g_{0})$ minimizes $p$-energy in its homotopy class among maps $F$ such that the average volume of $F(\C P^{2})$ for $\C P^{2} \in \mathcal{C}(\mathbb{H}P^{N})$ is at least $Vol(\C P^{2},g_{0}) = \frac{\pi^{2}}{2}$.  More generally, with more precise information about the minima or averages of the volumes of $F(\mathbb{H}P^{1})$ and $F(\C P^{2})$ for $\mathbb{H}P^{1} \in \mathcal{H}(\mathbb{H}P^{N})$ and $\C P^{2} \in \mathcal{C}(\mathbb{H}P^{N})$, one can deduce stronger lower bounds for energy functionals of $F$. \\ 

One possible approach to investigating the $p$-energy of mappings homotopic to the identity of $\mathbb{H} P^{N}$ is to use the homotopy lifting property to consider the family of mappings $\widetilde{F}: \C P^{2N+1} \rightarrow \C P^{2N+1}$ which cover a mapping $F: \mathbb{H} P^{N} \rightarrow \mathbb{H} P^{N}$ via the twistor fibration.  In particular, this may be helpful in finding optimal lower bounds for the $p$-energy of mappings homotopic to the identity mapping of $\mathbb{H} P^{N}$ and determining when the identity is $p$-energy minimizing in its homotopy class.  



\begin{thebibliography}{ABCDE}

	
	\bibitem[BKSW09]{BKSW} Victor Bangert, Mikhail Katz, Steven Shnider and Shmuel Weinberger: {\em E7, Wirtinger Inequalities, Cayley 4-form and Homotopy}, Duke Mathematical Journal 146.1 (2009): 35-70.
	
	\bibitem[Be12]{Be1} Arthur L. Besse: {\em Manifolds All of Whose Geodesics are Closed}, Springer Science and Business Media, 2012.
	
	\bibitem[Be07]{Be2} Arthur L. Besse: {\em Einstein Manifolds}, Springer Science and Business Media, 2007.
	
	\bibitem[BBdeBR89]{BBdBR1} D. Burns, F. Burstall, P. de Bartolomeis and J. Rawnsley: {\em Stability of Harmonic Maps of K\"ahler Manifolds}, Journal of Differential Geometry, 30(2), 579-594.
	
	\bibitem[Cr87]{Cr1} Christopher B. Croke: {\em Lower Bounds on the Energy of Maps}, Duke Mathematical Journal 55.4 (1987): 901-908.
	
	\bibitem[CK03]{CK1} Christopher B. Croke and Mikhail Katz: {\em Universal Volume Bounds in Riemannian Manifolds}, Surveys in Differential Geometry 8.1 (2003): 109-137.
	
	\bibitem[EL88]{EL2} James Eells and Luc Lemaire: {\em Another Report on Harmonic Maps}, Bulletin of the London Mathematical Society 20.5 (1988): 385-524.
	
	\bibitem[ES64]{ES1} James Eells and Joseph H. Sampson: {\em Harmonic Mappings of Riemannian Manifolds}, American Journal of Mathematics 86.1 (1964): 109-160.
	
	\bibitem[Gn62]{Gn} Leon W. Green: {\em Auf Wiedersehensfl\"achen}, Annals of Mathematics (1963): 289-299. 
	
	\bibitem[GG81]{GG1} Detlef Gromoll and Karsten Grove: {\em On Metrics on $S^{2}$ All of Whose Geodesics are Closed}, Inventiones Mathematicae 65.1 (1981): 175-177.
	
	\bibitem[Gr81]{Gr2} Mikhail Gromov: {\em Structures M\'etriques Pour Les Vari\'et\'es Riemanniennes}, Textes Mathematiques 1 (1981).
	
	\bibitem[Gr83]{Gr3} Mikhail Gromov: {\em Filling Riemannian Manifolds}, Journal of Differential Geometry 18.1 (1983): 1-147.

	\bibitem[Gr07]{Gr1} Mikhail Gromov: {\em Metric Structures for Riemannian and Non-Riemannian Spaces}, Springer Science and Business Media, 2007.
	
	\bibitem[GR65]{GR_0} Robert C. Gunning and Hugo Rossi: {\em Analytic Functions of Several Complex Variables}, Vol. 368. American Mathematical Soc., 2009.
	
	\bibitem[HL82]{HL1} Reese Harvey and H. Blaine Lawson, Jr.: {\em Calibrated Geometries}, Acta Mathematica 148 (1982): 47-157.

	\bibitem[Kr65]{Kr1} Vivian Y. Kraines: {\em Topology of Quaternionic Manifolds}, Bull. Am. Math. Soc. 71.3. P1 (1965): 526-527.
	
	\bibitem[Kr66]{Kr2} Vivian Y. Kraines: {\em Topology of Quaternionic Manifolds}, Trans. Am. Math. Soc. 122.2 (1966): 357-367.
	
	\bibitem[Li70]{Li1} Andr\'e Lichnerowicz: {\em Applications Harmoniques et Vari\'et\'es K\"ahleriennes}, Symp. Math. III, Bologna (1970), 341-402. 
	
	\bibitem[LS17]{LS1} Samuel Lin and Benjamin Schmidt: {\em Real Projective Spaces With All Geodesics Closed}, Geometric and Functional Analysis 27.3 (2017): 631-636.
	
	\bibitem[Oh87]{Oh1} Yoshihiro Ohnita: {\em On Pluriharmonicity of Stable Harmonic Maps}, Journal of the London Mathematical Society 2.3 (1987): 563-568.

	\bibitem[Pu52]{Pu} Pao Ming Pu: {\em Some Inequalities in Certain Nonorientable Riemannian Manifolds}, Pacific J. Math 2.1 (1952): 55-71.
	
	\bibitem[RW17]{RW1} Marco Radeschi and Burkhard Wilking: {\em On the Berger Conjecture for Manifolds All of Whose Geodesics Are Closed}, Inventiones Mathematicae 210.3 (2017): 911-962.
	
	\bibitem[SU81]{SU} Jonathan Sacks and Karen Uhlenbeck: {\em The Existence of Minimal Immersions of 2-spheres}, Annals of Mathematics (1981): 1-24.
	
	\bibitem[Sa82]{Sa1} Simon Salamon: {\em Quaternionic K\"ahler Manifolds}, Inventiones Mathematicae 67.1 (1982): 143-171.

	\bibitem[Sm75]{Sm1} Robert T. Smith: {\em The Second Variation Formula for Harmonic Mappings}, Proceedings of the American Mathematical Society (1975): 229-236.
	

	\bibitem[We98]{We1} Shishu W. Wei: {\em Representing Homotopy Groups and Spaces of Maps by p-Harmonic Maps},  Indiana University Mathematics Journal (1998): 625-670.

	\bibitem[Wh84]{Wh2} Brian White: {\em Mappings that Minimize Area in Their Homotopy Classes}, Journal of Differential Geometry 20.2 (1984): 433-446.

	\bibitem[Wh86]{Wh1} Brian White: {\em Infima of Energy Functionals in Homotopy Classes of Mappings}, Journal of Differential Geometry 23.2 (1986): 127-142.
	
	\bibitem[Wh88]{Wh3} Brian White: {\em Homotopy Classes in Sobolev Spaces and the Existence of Energy Minimizing Maps}, Acta Mathematica 160.1 (1988): 1-17.
		
\end{thebibliography}
\end{document}